\documentclass[hidelinks,onefignum,onetabnum]{siamart251216}

\usepackage{lipsum}
\usepackage{amsfonts}
\usepackage{graphicx}
\usepackage{epstopdf}
\usepackage{algorithmic}
\ifpdf
  \DeclareGraphicsExtensions{.eps,.pdf,.png,.jpg}
\else
  \DeclareGraphicsExtensions{.eps}
\fi

\newsiamremark{remark}{Remark}
\newsiamremark{hypothesis}{Hypothesis}
\crefname{hypothesis}{Hypothesis}{Hypotheses}
\newsiamthm{claim}{Claim}
\newsiamremark{fact}{Fact}
\crefname{fact}{Fact}{Facts}

\headers{Input-to-state stability of chemical reaction networks}{R. Jiang, X. Zhang, C. Gao, D. Dochain}

\title{Input-to-state stability of chemical reaction networks with application to molecular computation\thanks{Submitted to the editors DATE.
\funding{This work was funded by the National Natural Science Foundation of China under Grant No. 12320101001, and the National Foreign Expert Project of China under Grant No. S20250211.}}}

\author{Renlei Jiang\thanks{School of Mathematical Sciences, Zhejiang University, Hangzhou, China (\email{jiangrl@zju.edu.cn}).} \and Xiaoyu Zhang\thanks{School of Mathematics, Southeast University, Nanjing, China (\email{Xiaoyu\_Z@seu.edu.cn}).} \and Chuanhou Gao\thanks{Corresponding author. School of Mathematical Sciences and Center for Interdisciplinary Applied Mathematics, Zhejiang University, Hangzhou, China (\email{gaochou@zju.edu.cn}).} \and Denis Dochain\thanks{ICTEAM, UCLouvain, B\^{a}timent Euler, avenue Georges Lema\^{i}tre 4-6, 1348 Louvain-la-Neuve, Belgium (\email{denis.dochain@uclouvain.be}).}
}

\usepackage{amsopn}

\ifpdf
\hypersetup{
  pdftitle={An Example Article},
  pdfauthor={D. Doe, P. T. Frank, and J. E. Smith}
}
\fi

\usepackage{amsmath, amssymb, mathrsfs}
\usepackage[version=4]{mhchem}
\usepackage{tikz}
\usepackage{graphicx}
\usepackage{enumitem}

\newsiamremark{example}{Example}

\begin{document}

\maketitle

\begin{abstract}
Input-to-state stability (ISS) provides a useful tool for analyzing the robustness of time-varying chemical reaction networks (CRNs). This paper investigates the ISS property for CRNs in two cases: one is for weakly reversible CRNs with nonzero deficiency and multiple linkage classes, which is beyond the limit of existing results requesting a weakly reversible CRN to be zero-deficiency and single-linkage-class; the other is for non-weakly reversible CRNs by leveraging the notion of linear conjugacy. These results, on the one hand enrich the studies of the ISS property of broader classes of CRNs, on the other hand serve for designing CRN-based molecular computing systems. The latter application suggests the ISS theory can justify the parallel mechanism of biomolecular computations. We use some CRNs of practical relevance to demonstrate our results, including the p53 signaling and dimerization network, the p21-activated kinase 1 network, etc. 

\end{abstract}




\section{Introduction}
Chemical reaction network (CRN) is a fundamental mathematical framework for describing chemical reactions and studying biological processes. It has been extensively studied in both deterministic and stochastic formulations \cite{anderson2015lyapunov,hoessly2021stationary}. When a CRN is endowed with mass-action kinetics, its dynamics are represented by a system of polynomial ordinary differential equations, commonly referred to as a mass-action system (MAS). The dynamical analysis of MASs has been extensively studied over the past several decades, with particular attention devoted to the roles of network structure \cite{craciun2025structure,feinberg1987chemical}, equilibria \cite{boros2019existence}, stability \cite{horn1972general}, and persistence \cite{angeli2011persistence}. Among these studies, the most influential results are the Deficiency Zero Theorem \cite{feinberg1987chemical} and complex balanced theory \cite{horn1972general}, which establish strong relationships between weak reversibility, deficiency, and the existence and stability of positive equilibria. 

However, most of the above studies focused on the case in which the reaction rate constants are fixed. Relatively little attention has been paid to time-varying systems in which the reaction rate constants change over time \cite{angeli2011persistence,chaves2005input}, despite the fact that, in practice, such constants may be affected by external factors such as temperature \cite{angeli2011persistence}. Input-to-state stability (ISS) provides a natural framework for analyzing the robustness of such systems with time-varying inputs. Roughly speaking, the ISS property states that, for a bounded input $u$, the trajectory can be bounded by a function of $u$. Furthermore, as $t$ increases, the trajectory approaches (in a Lyapunov stability manner) a ball whose radius depends on $u$ \cite{sontag1989smooth}. In the context of rate-controlled biochemical networks, Chaves and Sontag established semiglobal ISS results for a class of weakly reversible CRNs that have zero deficiency and a single linkage class, and applied these results to observer design \cite{chaves2005input,chaves2002state}.

As the development of silicon-based computation is increasingly constrained by physical limitations, many researchers have turned their attention to biomolecular computation because of its advantages, including low energy consumption, parallel computing, and biocompatibility \cite{chen2024synthetic,cherry2025supervised}. Many of these theoretical studies are based on the mathematical framework of CRNs \cite{fan2025automatic,vasic2022programming}. Recently, several studies have explored the possibility of operating multiple modules of biomolecular computations in parallel \cite{anderson2025chemical,chalk2019composable}, some of which make use of the ISS concept.

Despite these important results, the existing ISS theory for CRNs remains limited in several respects. First, many CRNs of interest have nonzero deficiency and multiple linkage classes. Second, a large class of biochemical and synthetic reaction networks is not weakly reversible. For these systems, classical complex balanced theory cannot be directly applied. Moreover, in applications to molecular computations, a bounded trajectory is not sufficient. It also needs to establish that a downstream system converges to the equilibrium associated with the limiting value of an upstream input signal. These issues motivate the development of ISS conditions for broader classes of CRNs and the investigation of their implications for interconnected molecular reaction modules.

In this paper, we extend the results of Chaves and Sontag to CRNs with more general structures, including networks with nonzero deficiency, multiple linkage classes, and networks that are not weakly reversible. The contributions of this paper are summarized as follows:
\begin{itemize}
    \item To establish the semiglobal ISS property for weakly reversible networks, we use the toric locus introduced in \cite{craciun2025structure} to characterize the admissible input-value set.
    \item To establish the semiglobal ISS property for some non-weakly reversible networks, we exploit the concept of linear conjugacy \cite{johnston2011linear} to relate their dynamical properties to those of weakly reversible networks, thereby deriving the corresponding admissible input-value sets.
    \item Based on the above results, we show that when the reaction-rate input converges to a constant value, the state variable also converges to an equilibrium associated with the limiting input, thereby providing a theoretical foundation for the implementation of parallel molecular computations.
\end{itemize}

The rest of this paper is organized as follows. Section \ref{section_2} introduces some preliminaries on CRNs and MASs, and presents the motivation of this paper. Section \ref{section_3} reviews several existing concepts and results, and establishes the semiglobal ISS property for weakly reversible networks. Section \ref{section_4} establishes the semiglobal ISS property for non-weakly reversible networks. In Section \ref{section_5}, we discuss the application of the established results to the implementation of parallel molecular computations. Finally, Section \ref{section_6} concludes this paper.

\textbf{Mathematical Notation:} $\mathbb{R}^n,\mathbb{R}_{\ge0}^n,\mathbb{R}_{>0}^n,\mathbb{Z}_{\ge 0}^n,\mathbb{Z}_{> 0}^n$ represent $n$-dimensional real space, nonnegative real space, positive real space, nonnegative integer space, and positive integer space, respectively; $x^{v_{\cdot j}}\triangleq \prod_{i=1}^nx_i^{v_{ij}}$, where $x \in \mathbb{R}^n$, $v_{\cdot j}\in \mathbb{Z}^n$ and $0^0=1$; $|\cdot|$ denotes the Euclidean norm; ess.sup. represents the essential supremum; $\|u(t)-u^*\|\triangleq \mathrm{ess.sup.}\{|u(t)-u^*|:t\ge 0\}$; $\gamma : \mathbb{R}_{\ge 0} \to \mathbb{R}_{\ge 0}$ is a class $\mathcal{K}$ function if it is continuous, strictly increasing and satisfies $\gamma(0)=0$; $\gamma$ is a class $\mathcal{K}_\infty$ function if it is a class $\mathcal{K}$ function and satisfies $\lim_{s \to \infty}\gamma(s)=\infty$; $\beta : \mathbb{R}_{\ge 0} \times \mathbb{R}_{\ge 0} \to \mathbb{R}_{\ge 0}$ is a class $\mathcal{KL}$ function if for each fixed $t$ the mapping $\beta(\cdot,t)$ is a class $\mathcal{K}$ function and for each fixed $s$ the function $\beta(s,t)$ decreases to zero on $t$ as $t \to \infty$.



\section{Preliminaries and motivation}\label{section_2}
In this section, we introduce some basic knowledge about CRN and MAS \cite{feinberg2019foundations}, and then present the motivation for the current study. 

\subsection{CRN}
Consider a CRN with $n$ species, denoted by $X_1,...,X_n$, and $r$ reactions with the $j$th reaction written as
$$\sum_{i=1}^nv_{ij}X_i\to \sum_{i=1}^nv'_{ij}X_i,$$
where $v_{ij}$ represents the stoichiometric coefficient of species $X_i$ 
in the $j$th reaction. Furthermore, define $v_{.j}=(v_{1j},...,v_{nj})^{\top}, v'_{.j}=(v_{1j}^{\prime},...,v_{nj}^{\prime})^{\top} \in\mathbb{Z}_{\geq 0}^n$ 
as the reactant and product complexes of the $j$th reaction, respectively.
The reaction can then be written as $v_{.j}\to v'_{.j}$.
Mathematically, a CRN is defined as follows.  

\begin{definition}[CRN]\label{def_CRN}
    A CRN consists of three finite sets: 
    \begin{enumerate}
        \item[\textup{(i)}] a finite \textit{species} set $\mathcal{S}=\{ X_1,...,X_n\}$;
        \item[\textup{(ii)}] a finite \textit{complex} set $\mathcal{C}=\bigcup_{j=1}^r{\left\{ v_{\cdot j},v_{\cdot j}^{\prime} \right\}}$;
        \item[\textup{(iii)}] a finite \textit{reaction} set $\mathcal{R}=\bigcup_{j=1}^r{\left\{ v_{\cdot j}\rightarrow v_{\cdot j}^{\prime} \right\}}$ satisfying
        \begin{enumerate}
            \item $\forall~ v_{\cdot j}\rightarrow v_{\cdot j}^{\prime}\in \mathcal{R},~ v_{\cdot j}\ne v_{\cdot j}^{\prime}$,
            \item $\forall~ v_{\cdot j} \in \mathcal{C}, \exists~ v_{\cdot j}^{\prime} \in \mathcal{C}$ such that $v_{\cdot j}\rightarrow v_{\cdot j}^{\prime}\in \mathcal{R}$ or $v_{\cdot j}^{\prime}\rightarrow v_{\cdot j}\in \mathcal{R}$.
        \end{enumerate}
    \end{enumerate}
    The triple $\mathcal{N}=(\mathcal{S},\mathcal{C},\mathcal{R})$ is usually used to express a CRN.
\end{definition}

A CRN can be represented as a directed graph whose vertices correspond to 
complexes and whose directed edges correspond to reactions. This graph-theoretic
representation allows several important structural classes of CRNs to be defined.

\begin{definition}[Reversible/Weakly reversible CRN]
    A CRN $(\mathcal{S},\mathcal{C},\mathcal{R})$ is 
    \begin{enumerate}
        \item[\textup{(i)}] \textit{reversible} if $\forall~ v_{\cdot j} \to v_{\cdot j}^{\prime} \in \mathcal{R}$, it holds $v_{\cdot j}^{\prime} \to v_{\cdot j} \in \mathcal{R}$;
        \item[\textup{(ii)}] \textit{weakly reversible} if $\forall~ v_{\cdot j} \to v_{\cdot j}^{\prime} \in \mathcal{R}$, there exists a sequence of complexes $v_{\cdot j_1},\cdots ,v_{\cdot j_p} \in \mathcal{C}$ such that $v_{\cdot j}^{\prime} \to v_{\cdot j_1} \in \mathcal{R}, v_{\cdot j_1} \to v_{\cdot j_2} \in \mathcal{R}$, $\cdots$, $v_{\cdot j_{p-1}} \to v_{\cdot j_p} \in \mathcal{R}$, $v_{\cdot j_p} \to v_{\cdot j} \in \mathcal{R}$.
    \end{enumerate}
\end{definition}

Clearly, every reversible CRN is weakly reversible, but not vice versa. A CRN is weakly reversible if and only if each of its \texttt{linkage class}, defined as a connected component of the underlying undirected reaction
graph, is strongly connected, or equivalently, if every reaction lies on a
directed cycle. 

For the reaction $v_{\cdot j}\rightarrow v_{\cdot j}^{\prime}$ in the CRN, we call $v_{\cdot j}^{\prime}-v_{\cdot j}$ its \texttt{reaction vector}, all of which induce two important concepts for dynamic analysis.

\begin{definition}[Stoichiometric subspace]\label{def_stoichiometric_subspace}
    For a CRN $(\mathcal{S},\mathcal{C},\mathcal{R})$, the linear subspace $\mathscr{S}\triangleq\mathrm{span}\{ v_{\cdot 1}^{\prime}-v_{\cdot 1},...,v_{\cdot r}^{\prime}-v_{\cdot r}\}$ is called the \textit{stoichiometric subspace} of the network. The dimension of $\mathscr{S}$, $\dim \mathscr{S}$, is called the dimension of the CRN.
\end{definition}

\begin{definition}[Stoichiometric compatibility class]
For a CRN $(\mathcal{S},\mathcal{C},\mathcal{R})$, let $x_0 \in \mathbb{R}_{\ge0}^n$, the set $x_0+\mathscr{S}=\{x_0+x:x \in \mathscr{S}\}$ is called the \textit{stoichiometric compatibility class} of $x_0$. Further, $(x_0+\mathscr{S}) \bigcap \mathbb{R}_{\ge 0}^n$ and $(x_0+\mathscr{S}) \bigcap \mathbb{R}_{> 0}^n$ are called the \textit{nonnegative stoichiometric compatibility class} and the \textit{positive stoichiometric compatibility class} of $x_0$, respectively.
\end{definition}

Based on these notions, we further introduce the concept of \texttt{deficiency}, which is often used to help characterize dynamical behaviors.

\begin{definition}[Deficiency]
    For a CRN $(\mathcal{S},\mathcal{C},\mathcal{R})$, let $c$ denote the number of complexes and $\ell$ denote the number of linkage classes. The \textit{deficiency} of the CRN is defined by $\delta =c-\ell-\dim \mathscr{S}$.
\end{definition}

The deficiency of a CRN is usually nonnegative because it can be seen as the dimension of a certain linear subspace \cite{feinberg2019foundations}. To better understand the above abstract concepts, we illustrate them with a concrete example.

\begin{example}\label{ex_CRN-MAS}
    For the following CRN $(\mathcal{S},\mathcal{C},\mathcal{R})$ shown on the left-hand side
    \begin{equation*}\label{eq_exMAS}
        \begin{tikzpicture}
        \node (x) at (-1,0.85) {$\mathrm{CRN:}$};
	    \node (a) at (0,0) {$2X_1$};
	    \node (b) at (2,0) {$X_2$};
	    \node (c) at (1,1.7) {$X_2+X_3$};
	    \draw[->] (a) -- (b);
   	  \draw[->] (b) -- (c);
    	\draw[->] (c) -- (a);

        \draw[dashed] (4,-0.5) -- (4,1.9);

        \node (x2) at (6,0.85) {$\mathrm{MAS:}$};
        \node (a2) at (7,0) {$2X_1$};
	    \node (b2) at (9,0) {$X_2$};
	    \node (c2) at (8,1.7) {$X_2+X_3$};
	    \draw[->] (a2) -- (b2) node[midway, below] {$\kappa_1$};
   	  \draw[->] (b2) -- (c2) node[midway, right] {$\kappa_2$};
    	\draw[->] (c2) -- (a2) node[midway, left] {$\kappa_3$};
       \end{tikzpicture}
    \end{equation*}
we get $\mathcal{S}=\{X_1,X_2,X_3\},~\mathcal{C}=\{(2,0,0)^{\top},(0,1,0)^{\top},(0,1,1)^{\top}\},~\mathcal{R}=\{(2,0,0)^{\top} \to (0,1,0)^{\top},(0,1,0)^{\top} \to (0,1,1)^{\top},  (0,1,1)^{\top} \to (2,0,0)^{\top}\}$, $\mathscr{S}=\mathrm{span}\{(-2,1,0)^{\top},(0,0,1)^{\top},\\(2,-1,-1)^\top\}$, $\dim \mathscr{S}=2$, and $\delta=3-1-2=0$. Clearly, it is weakly reversible, but not reversible.
\end{example}

\subsection{MAS}
When a CRN $(\mathcal{S},\mathcal{C},\mathcal{R})$ is equipped with mass-action kinetics, the rate of reaction $v_{\cdot j} \to v_{\cdot j}^{\prime}$ is measured by $\kappa_jx^{v_{\cdot j}}$, where $\kappa_j>0$ represents the rate constant, and $x \in \mathbb{R}_{\ge 0}^n$ with each element $x_i~(i=1,...,n)$ to represent the concentration of the species $X_i$.

\begin{definition}[MAS]
 A \textit{MAS} is a CRN $(\mathcal{S},\mathcal{C},\mathcal{R})$ equipped with mass-action kinetics $\kappa=(\kappa_1,... ,\kappa_r)^{\top}$, often labeled by $(\mathcal{S},\mathcal{C},\mathcal{R},\kappa)$ or $(\mathcal{N},\kappa)$. 
\end{definition}

The dynamics of $(\mathcal{S},\mathcal{C},\mathcal{R},\kappa)$ describes the change of concentrations of all species over time $t$, and thus follows
\begin{equation}\label{general dynamics}
    \dot{x}=\sum_{j=1}^{r}\kappa_jx^{v_{\cdot j}}\left(v_{\cdot j}^{\prime}-v_{\cdot j} \right),
\end{equation}
which are essentially polynomial ODEs. By integrating (\ref{general dynamics}) from $0$ to $t$, we get
\begin{equation}\label{integ_dynamics}
 x(t) = x_0 + \sum_{j=1}^r \left(v_{\cdot j}^{\prime}-v_{\cdot j} \right) \int_0^t \kappa_jx^{v_{\cdot j}}(\tau)\mathrm{d}\tau,   
\end{equation}
where $x_0=x(0)$ is the initial state of $(\mathcal{S},\mathcal{C},\mathcal{R},\kappa)$. This suggests that the state of $(\mathcal{S},\mathcal{C},\mathcal{R},\kappa)$ will evolve in the nonnegative stoichiometric compatibility class of $x_0$, i.e., in $(x_0+\mathscr{S}) \bigcap \mathbb{R}_{\ge 0}^n$. We can also write the equivalent expression of (\ref{general dynamics}) as a sum over complexes 
\begin{equation}\label{general dynamics_complex}
    \dot{x}=\sum_{z\in\mathcal{C}}x^z\sum_{\{j|v_{.j}=z\}}\kappa_j\left(v_{\cdot j}^{\prime}-v_{\cdot j} \right).
\end{equation}

Revisiting \textit{Example} \ref{ex_CRN-MAS} reveals that the dynamics of the MAS (on the right-hand side of the diagram in \textit{Example} \ref{ex_CRN-MAS}) is given by
    \begin{equation*}
        \left ( \begin{array}{c}
         \dot{x}_1 \\ \dot{x}_2 \\ \dot{x}_3
    \end{array} \right) =
     \left ( \begin{array}{c}
         -2\kappa_1x_1^2+2\kappa_3 x_2x_3 \\ \kappa_1x_1^2-\kappa_3x_2x_3 \\ \kappa_2x_2-\kappa_3x_2x_3
    \end{array} \right).
    \end{equation*}

\begin{definition}[Equilibrium]
     For a MAS $(\mathcal{S},\mathcal{C},\mathcal{R},\kappa)$ governed by (\ref{general dynamics}), a constant vector $x^* \in \mathbb{R}^n_{>0}$ is called a \textit{positive equilibrium} of the system if
    \begin{equation}
        \sum_{j=1}^{r}\kappa_j(x^*)^{v_{\cdot j}}\left(v_{\cdot j}^{\prime}-v_{\cdot j} \right)=0.
    \end{equation}
    A MAS that admits a positive equilibrium is said to be \textit{balanced}.
\end{definition}

\begin{definition}[Complex balanced equilibrium]\label{def_cbMAS}
    For a MAS $(\mathcal{S},\mathcal{C},\mathcal{R},\kappa)$ governed by (\ref{general dynamics}), a positive equilibrium $x^* \in \mathbb{R}^n_{>0}$ is called a \textit{complex balanced equilibrium} of the system if
    \begin{equation}
        \sum_{\{j|v_{\cdot j}=z\}}\kappa_j(x^*)^{v_{\cdot j}}=\sum_{\{j|v_{\cdot j }^{\prime}=z\}}\kappa_j(x^*)^{v_{\cdot j}},\quad \forall z \in \mathcal{C}.
    \end{equation}
    A MAS that admits a complex balanced equilibrium is said to be \textit{complex balanced}.
\end{definition}

The complex balanced equilibrium is certainly an equilibrium, but not vice versa. There are some well-known results about complex balanced MAS \cite{feinberg1987chemical,horn1972general}, including 
\begin{itemize}
 \item [(1)] a complex balanced MAS must be weakly reversible;
\item [(2)] if a MAS has a complex balanced equilibrium, then all the other equilibria (if any) are complex balanced;
          \item[(3)] Deficiency Zero Theorem: for any $\kappa \in \mathbb{R}^{r}_{>0}$, the MAS $(\mathcal{S},\mathcal{C},\mathcal{R},\kappa)$ is complex balanced if it is weakly reversible and has zero deficiency;
        \item [(4)] within each positive stoichiometric compatibility class, there is precisely one complex balanced equilibrium, and moreover, each complex balanced equilibrium is locally asymptotically stable. 
\end{itemize}

For the last result, the pseudo-Helmholtz free energy function
    \begin{equation}\label{Lyafun_deficiency_zero}
        V(x)=\sum_{i=1}^n \left(x_i(\ln x_i-\ln x_i^*-1)+x^*_i\right),~x \in \mathbb{R}^{n}_{> 0}
    \end{equation}
    is suggested as a Lyapunov function to establish the local asymptotic stability of each complex balanced equilibrium $x^*$.

\subsection{Motivation}
In the theoretical analysis of an MAS $(\mathcal{S},\mathcal{C},\mathcal{R},\kappa)$, the reaction rate constants $\kappa$ are typically treated as fixed parameters \cite{craciun2025structure,feinberg1987chemical,horn1972general,johnston2011linear}. However, under realistic biological conditions, these parameters often vary over time in response to external inputs, like temperature fluctuations, external electric-field effects \cite{angeli2011persistence}, uncertainty arising from the dynamics of unmodeled species \cite{vaghy2023persistence}, or the encoding of species concentrations as the inputs for molecular computation \cite{jiang2025input,jiang2026structure}. We thus consider a time-varying version of MAS governed by
\begin{equation}\label{eq_time_varying_CRN}
    \dot{x}=f(x;u)=\sum_{j=1}^r u_j(t)x^{v_{\cdot j}}\left(v_{\cdot j}^\prime-v_{\cdot j}\right),
\end{equation}
which is also called rate-controlled MAS in ISS capture \cite{chaves2005input}. Here, $x \in \mathbb{R}^n_{\ge 0},~u:\mathbb{R}_{\ge 0} \to \mathbb{U}$ with $\mathbb{U} \subseteq \mathbb{R}^r_{>0}$ is a piecewise locally Lipschitz function. Notably, if the system starts from an initial point $x_0 \in \mathbb{R}^n_{\ge 0}$, then the state of (\ref{eq_time_varying_CRN}) will still evolve in $(x_0+\mathscr{S})\cap \mathbb{R}^n_{\ge 0}$.

For the purpose of capturing ISS and its application to molecular computation, there are two questions of concern about (\ref{eq_time_varying_CRN}):

\textbf{Question 1:} For a bounded input $u(t)$, can the solution $x(t)$ of (\ref{eq_time_varying_CRN}) be bounded by an increasing function of $u$?

\textbf{Question 2:} If $\lim_{t \to \infty}u(t)=u^*$, can we have $\lim_{t \to \infty}x(t)=x^*$, where $x^*$ satisfies $f(x^*;u^*)=0$?

For \textbf{Question 1}, Chaves and Sontag \cite{chaves2005input,chaves2002state} have addressed it through formalizing the concept of semiglobal ISS (see \cref{def_ISS} below). They also proved a class of special weakly reversible CRNs with zero deficiency and single linkage class to be semiglobal ISS. Building on this concept, we further addressed \textbf{Question 2}, and applied the result to enable serial computing by parallel chemical reactions \cite{jiang2025input,jiang2026structure}. The strong application potential of the notion of semiglobal ISS in molecular computations motivates us to find more CRNs (not limited to single-linkage-class, deficiency-zero, weakly reversible CRNs) that can exhibit semiglobal ISS. 

In what follows, we proceed by reviewing the existing result for the case of deficiency-zero, single-linkage-class, weakly reversible CRNs \cite{chaves2005input,chaves2002state}, and then to the cases of nonzero-deficiency weakly reversible CRNs and non-weakly reversible CRNs.


\section{ISS for weakly reversible CRNs}\label{section_3}
In this section, we focus on exploring the ISS property for weakly reversible CRNs, including revisiting existing result for zero-deficiency weakly reversible CRNs, and then extending them to weakly reversible CRNs with nonzero deficiency and multiple linkage classes.

\subsection{Existing result for zero-deficiency, weakly reversible CRNs}
We begin by recalling several definitions \cite{chaves2005input,chaves2002state} related to ISS for MASs. 

\begin{definition}
 For a time-varying MAS of (\ref{eq_time_varying_CRN}) with input-value set $\mathbb{U} \subseteq \mathbb{R}^r_{>0}$, if for each initial state $x(0) \in \mathbb{R}^n_{>0}$ and each input $u(\cdot)\in \mathbb{U} $, the solution of (\ref{eq_time_varying_CRN}) at any time $t\in J_{x(0),u}=[0,t_{max})$ with $t_{max}$ to represent the blowing-up time, is $x(t) \in \mathbb{R}^n_{>0}$, then the system is $\mathbb{R}^n_{>0}$-forward invariant. Further, the system is $\mathbb{R}^n_{>0}$-forward complete if it is $\mathbb{R}^n_{>0}$-forward invariant with $J_{x(0),u}=[0,+\infty)$.
\end{definition}

\begin{definition}\label{def_ISS}
 The time-varying MAS of (\ref{eq_time_varying_CRN}) with input-value set $\mathbb{U} \subseteq \mathbb{R}^r_{>0}$ is called semiglobally input-to-state stable with respect to $(x^*,u^*)$ with fixed points $x^* \in \mathbb{R}^n_{>0}$ and $u^* \in \mathbb{U}$, if 
    \begin{enumerate}
        \item[\textup{(i)}] the system is $\mathbb{R}^n_{>0}$-forward complete;
        \item[\textup{(ii)}] for every compact set $F \subseteq \mathbb{R}^n_{\ge 0}$ containing $x^*$, and moreover, $\forall t \ge 0$ we have $x(s) \in F$ for any $s \in [0,t]$, there exist a class $\mathcal{KL}$ function $\beta=\beta_{F}$ and a class $\mathcal{K}_{\infty}$ function $\gamma=\gamma_F$ such that
        \begin{equation}
            |x(t)-x^*|\le \beta(|x_0-x^*|,t)+\gamma(\|u-u^*\|)
        \end{equation}
        for each initial condition $x_0=x(0) \in F \cap \{x^*+\mathscr{S}\}\cap \mathbb{R}^n_{>0}$ and each input $u(\cdot)\in \mathbb{U}$. 
    \end{enumerate}
\end{definition}

\begin{definition}\label{def_ISS_Lyafun}
Given a time-varying MAS of (\ref{eq_time_varying_CRN}) with input-value set $\mathbb{U} \subseteq \mathbb{R}^r_{>0}$, a continuous function $V:\mathbb{R}^n_{\ge 0} \to \mathbb{R}_{\ge 0}$ is a semiglobal ISS-Lyapunov function with respect to the fixed points $x^* \in \mathbb{R}^n_{>0}$ and $u^* \in \mathbb{U}$ for the system, if
    \begin{enumerate}
        \item[\textup{(i)}] the restriction of $V$ to $\mathbb{R}^n_{>0}$ is continuously differentiable;
        \item[\textup{(ii)}] there exist class $\mathcal{K}_{\infty}$ functions $\underline{\alpha},\overline{\alpha}$  such that
        $$\underline{\alpha}(|x-x^*|)\le V(x) \le \overline{\alpha}(|x-x^*|)$$
        for each $x \in \mathbb{R}^n_{\ge 0}$;
        \item[\textup{(iii)}] for every compact set $F \subseteq \mathbb{R}^n_{\ge 0}$ containing $x^*$, there exist $\mathcal{K}_{\infty}$ functions $\alpha=\alpha_F,\rho=\rho_F$ such that
        $$\nabla^{\top} V(x)f(x;u)\le -\alpha(|x-x^*|)+\rho(|u-u^*|)$$
        for every $x \in F \cap \{x^*+\mathscr{S}\}\cap \mathbb{R}^n_{>0}$ and every $u \in \mathbb{U}$.
    \end{enumerate}
\end{definition}

Utilizing these concepts, Chaves and Sontag  \cite{chaves2002state} proved that for a $\mathbb{R}^n_{>0}$-forward complete time-varying MAS of (\ref{eq_time_varying_CRN}) with input-value set $\mathbb{U}$, if there is a semiglobal ISS-Lyapunov function $V$ with respect to the fixed point pair $(x^*,u^*)$ for the system, then it is semiglobal ISS with respect to $(x^*,u^*)$. As an application, they also showed that \textit{a zero-deficiency, single-linkage-class, weakly reversible CRN} is semiglobal ISS under the conditions of $f(x^*;u^*)=0$ and $\mathbb{U} \subseteq \mathbb{R}^r_{>0}$ being a specific set.

\subsection{ISS for nonzero deficiency weakly reversible CRNs}
Building on the above existing result, we try to capture semiglobal ISS for nonzero deficiency weakly reversible CRNs. 

\begin{lemma}\label{lem_ISS_Lya}
 For a time-varying MAS of (\ref{eq_time_varying_CRN}) with input-value set $\mathbb{U} \subseteq \mathbb{R}^r_{>0}$ to be a compact set, suppose that for any $u^* \in \mathbb{U}$, the MAS with $u(t)\equiv u^*$ has a unique equilibrium $x^*=x^*(u^*)$ in the positive stoichiometric compatibility class of $x^*$. Let $V_0:\mathbb{R}^n_{\ge 0} \times \mathbb{R}^r_{>0} \to \mathbb{R}_{\ge 0}$ be a continuous function satisfying:
    \begin{enumerate}
        \item[(1)] The restriction of $V_0$ to $\mathbb{R}^n_{>0} \times \mathbb{U}$ is continuously differentiable;
        \item[(2)] For any $u^* \in \mathbb{U}$, there exist class $\mathcal{K}_{\infty}$ functions $\underline{\alpha}_{u^*},\overline{\alpha}_{u^*}$ such that $$\underline{\alpha}_{u^*}(|x-x^*|) \le V_0(x,u^*) \le \overline{\alpha}_{u^*}(|x-x^*|),~x \in \mathbb{R}^n_{\ge 0};$$
        \item[(3)] For any $u^* \in \mathbb{U}$, it holds that $$\nabla_x^{\top} V_0(x,u^*) f(x;u^*)\le 0,~x \in \mathbb{R}^n_{>0}$$ with the equality holding if and only if $x=x^*$.
    \end{enumerate}
Then $V(x)\triangleq V_0(x,u^*)$ is a semiglobal ISS-Lyapunov function for the system (\ref{eq_time_varying_CRN}) with respect to $(x^*,u^*)$, where $x^*$ supports $f(x^*;u^*)=0$.
\end{lemma}

\begin{proof}
    According to condition (1) and (2), we know that $V(x)$ satisfies condition (i) and (ii) in \cref{def_ISS_Lyafun}. Therefore, we just need to show that $V(x)$ satisfies condition (iii). For any compact set $F \subseteq \mathbb{R}^n_{\ge 0}$ containing $x^*$, it holds
    \begin{equation}\label{eq_pf1_1}
    \begin{split}
        \nabla^{\top} V(x)f(x;u) & =\nabla_x^{\top}V_0(x,u^*)f(x;u^*)+\nabla_x^{\top}V_0(x,u^*)f(x;u-u^*) \\
        &\triangleq I_1+I_2
    \end{split}
    \end{equation}
    By condition (3) we can get $-I_1=-\nabla_x^{\top}V_0(x,u^*)f(x;u^*)\ge 0,~x\in F\cap (x^*+\mathscr{S}) \cap \mathbb{R}^n_{>0}$, with the equality holding if and only if $x=x^*$. Denote $s_{\mathrm{max}}=\sup\{s:\exists~x\in F\cap (x^*+\mathscr{S})\cap \mathbb{R}^n_{>0}$ such that $|x-x^*|\ge s \}$, and let
    \begin{equation*}
        \tilde{\alpha}_F(s)=\inf_{\substack{|x-x^*|\ge s \\ x\in F\cap (x^*+\mathscr{S})\cap \mathbb{R}^n_{>0}}}\left(-\nabla_x^{\top}V_0(x,u^*)f(x;u^*)\right),\quad 0\le s \le s_{\mathrm{max}}.
    \end{equation*}
Since $\tilde{\alpha}_F(0)=0$ and $\tilde{\alpha}_F(s)>0$ for all 
$s\in(0,s_{\max}]$, and $\tilde{\alpha}_F$ is continuous and nondecreasing on 
$[0,s_{\max}]$, there exists a class $\mathcal{K}$ function $\alpha_F$ such that
$\alpha_F(s)\le \tilde{\alpha}_F(s)/2$ for all $s\in[0,s_{\max}]$.
Moreover, $\alpha_F$ can be extended to $\mathbb R_{\ge 0}$ as a class $\mathcal{K}_\infty$
function, for instance by defining
$\alpha_F(s)=\alpha_F(s_{\max})+s-s_{\max}$ for $s>s_{\max}$. Thus the following equation holds
\begin{equation}\label{eq_pf1_2}
        -I_1 \ge \tilde{\alpha}_F(|x-x^*|)\ge \alpha_F(|x-x^*|),~\forall~x \in F\cap (x^*+\mathscr{S})\cap \mathbb{R}^n_{>0}.
    \end{equation}
    Since $F$ is a compact set and $\nabla_x^{\top}V_0(x,u^*)$ is continuous, there is a constant $C_F>0$ such that
    \begin{equation}\label{eq_pf1_3}
        \begin{split}
            I_2&=\nabla_x^{\top}V_0(x,u^*)\sum_{j=1}^r(u_j-u_j^*)x^{v_{\cdot j}}(v_{\cdot j}^{\prime}-v_{\cdot j}) \\
            & \le C_F \|u-u^*\|=\rho_F(\|u-u^*\|)
        \end{split}
    \end{equation}
    where $x \in F\cap (x^*+\mathscr{S})\cap \mathbb{R}^n_{>0}$, $\rho_F(s)=C_Fs$ is a class $\mathcal{K}_{\infty}$ function. Substituting (\ref{eq_pf1_2}) and (\ref{eq_pf1_3}) into (\ref{eq_pf1_1}) we know that the function $V(x)$ satisfies condition (iii) in \cref{def_ISS_Lyafun}, and the proof is complete.
\end{proof}

\begin{proposition}\label{thm_ISS_lya}
Consider a time-varying MAS of (\ref{eq_time_varying_CRN}) with input-value set $\mathbb{U} \subseteq \mathbb{R}^r_{>0}$ to be a compact set, and suppose that it is $\mathbb{R}^n_{>0}$-forward complete. Then for any $u^* \in \mathbb{U}$, under the conditions given in \cref{lem_ISS_Lya} the system (\ref{eq_time_varying_CRN}) starting from $x_0$ is semiglobal ISS with respect to $(x^*,u^*)$, where $x^*$ supports $f(x^*;u^*)=0$ and is contained in a compact set $F \subseteq \mathbb{R}^n_{\ge 0}$, and moreover, $x_0 \in F \cap \{x^*+\mathscr{S}\}\cap \mathbb{R}^n_{>0}$. 
\end{proposition}

\begin{proof}
   The proof follows directly by applying Lemma 3.7 in \cite{chaves2002state} and \cref{lem_ISS_Lya}.
\end{proof}

\cref{thm_ISS_lya} provides some conditions to support the MAS of (\ref{eq_time_varying_CRN}) to be semiglobal ISS. In particular, if the network is weakly reversible, we can get the main result in this section. For this purpose, we introduce an important concept that characterizes the set  $\mathbb{U}$.

\begin{definition}[Toric Locus \cite{craciun2025structure}]\label{def_toric_locus}
    For a CRN $\mathcal{N}=(\mathcal{S},\mathcal{C},\mathcal{R})$, its toric locus $\mathcal{T}(\mathcal{N}) \subseteq \mathbb{R}^r_{>0}$ denotes the set of parameters $\kappa \in \mathbb{R}^r_{>0}$, for which the corresponding MAS $(\mathcal{N},\kappa)$ is complex balanced.
\end{definition}

As discussed in \cite{craciun2025structure}, for a given CRN $\mathcal{N}$, its toric locus $\mathcal{T}(\mathcal{N}) \ne \varnothing$ if and only if $\mathcal{N}$ is weakly reversible. 

\begin{theorem}\label{coro_cb_mas}
Consider a time-varying MAS of (\ref{eq_time_varying_CRN}) with weakly reversible network structure. Denote its toric locus by $\mathcal{T}$, then for any compact input-value set $\mathbb{U} \subseteq \mathcal{T}$ and any $u^* \in \mathbb{U}$, the system (\ref{eq_time_varying_CRN}) starting from $x_0$ is semiglobal ISS with respect to $(x^*,u^*)$, where $x^*$ supports $f(x^*;u^*)=0$ and is contained in a compact set $F \subseteq \mathbb{R}^n_{\ge 0}$, and moreover, $x_0 \in F \cap \{x^*+\mathscr{S}\}\cap \mathbb{R}^n_{>0}$. 
\end{theorem}

\begin{proof}
We complete the proof in three steps.

    \textit{Step 1.} We prove that system (\ref{eq_time_varying_CRN}), with input-value set $\mathbb{U}$, is $\mathbb{R}^n_{>0}$-forward invariant, i.e., for each initial state $x(0)=x_0 \in \mathbb{R}^n_{>0}$ and each $\mathbb{U}$-valued input $u(\cdot)$, the corresponding maximal solution of (\ref{eq_time_varying_CRN}), which is defined on $[0,t_{max})$, satisfies $x(t) \in (x_0+\mathscr{S})\cap \mathbb{R}^n_{>0},~0\le t <t_{max}$. This proof is very similar to the proof of \textit{Proposition} 5.1 in \cite{chaves2005input},  so we will not reproduce it here.

    \textit{Step 2.} We prove that system (\ref{eq_time_varying_CRN}), with input-value set $\mathbb{U}$, is $\mathbb{R}^n_{>0}$-forward complete, i.e., $t_{max}=+\infty$. To prove this claim, we argue by contradiction. Suppose $t_{max}<+\infty$, then it holds $\lim_{t \to t_{max}}|x(t)|=+\infty$. Since $\mathbb{U} \subseteq \mathcal{T}$, for any $u\in \mathbb{U}$, the MAS is complex balanced, and thus there is precisely one equilibrium in each positive stoichiometric compatibility class. Now consider the function
    \begin{equation*}
        W(t)=V(x(t),x^*(u(t)))=\sum_{i=1}^n \left(x_i(t)(\ln x_i(t)-\ln x_i^*(u(t))-1)+x^*_i(u(t))\right),
    \end{equation*}
    where $x^*(u(t))$ is the unique equilibrium of $\dot{x}=f(x,u(t))$ in $(x_0+\mathscr{S})\cap \mathbb{R}^n_{>0}$. Clearly, it holds $\lim_{t \to t_{max}}W(t)=+\infty$, and for $t \in [0,t_{\max})$ we have
    \begin{equation}\label{eq_pf_0}
    \begin{split}
        \frac{dW}{dt}&=\nabla_x^{\top}V(x(t),x^*(u(t)))\dot{x}+\nabla^{\top}_{x^*}V(x(t),x^*(u(t)))\dot{x}^* \\
        &=\left(\mathrm{Ln}(x)-\mathrm{Ln}(x^*)\right)^{\top}f(x,u)+\sum_{i=1}^n\frac{1}{x_i^*}(x_i^*-x_i)\dot{x}_i^*
    \end{split}
    \end{equation}
    where $\mathrm{Ln}(x)=(\ln x_1,...,\ln x_n)^{\top},~\dot{x}^*=\frac{d}{dt}(x^*(u(t)))$. Since (\ref{Lyafun_deficiency_zero}) is the Lyapunov function of complex balanced MAS, we have
    \begin{equation}\label{eq_pf_1}
        \left(\mathrm{Ln}(x)-\mathrm{Ln}(x^*)\right)^{\top}f(x,u)\le 0.
    \end{equation}
    In addition, according to the \textit{Corollary} 3.13 in \cite{craciun2025structure}, the function $x^*=x^*(u)$ is continuously differentiable. As $u(t) \in \mathbb{U}$, where $u(t)$ is piecewise locally Lipschitz function, and $\mathbb{U}$ is a compact set, we have
    \begin{equation}\label{eq_pf_2}
        \dot{x}_i^*=\sum_{j=1}^r\frac{dx_i^*}{du_j}\frac{du_j}{dt} \leq \sum_{j=1}^r\left\vert\frac{dx_i^*}{du_j}\right\vert\cdot\left\vert\frac{du_j}{dt}\right\vert \le \sum_{j=1}^rC_1\cdot L=C_1Lr,
    \end{equation}
    and 
    \begin{equation}\label{eq_pf_3}
        1/x_i^*\le C_2,\quad x_i^* \le C_2
    \end{equation}
    where $C_1,C_2,L>0$ are constants. Finally, to estimate $x_i^*-x_i$, we define the function
    \begin{equation*}
        \varphi(r)=r(\ln r-\ln x_i^*-1)+2x_i^*-|x_i^*-r|.
    \end{equation*}
    For $r\ge x_i^*$, we have $\varphi^{\prime}(r)=\ln r-\ln x_i^*-1$, and $\varphi^{\prime}(r)>0 \Longleftrightarrow r>ex_i^*$, therefore it holds $\varphi(r)\ge \varphi(ex_i^*)=(3-e)x_i^*>0$; For $0<r<x_i^*$, we have $\varphi^{\prime}(r)=\ln r-\ln x_i^*+1$, and $\varphi^{\prime}(r)>0 \Longleftrightarrow x_i^*/e<r<x_i^*$, therefore it holds $\varphi(r)\ge \varphi(x_i^*/e)=(1-1/e)x_i^*>0$. In a word, we have $\varphi(r)>0,~\forall r>0$. Specially, $\varphi(x_i)>0$, which means
    \begin{equation}\label{eq_pf_4}
        |x_i^*-x_i|<x_i(\ln x_i-\ln x_i^*-1)+2x_i^*.
    \end{equation}
    Substituting \cref{eq_pf_1,eq_pf_2,eq_pf_3,eq_pf_4} into (\ref{eq_pf_0}) we obtain
    \begin{equation}\label{eq_pf_5}
        \begin{split}
            \frac{dW}{dt}&\le C_1C_2Lr\sum_{i=1}^n\left(x_i(\ln x_i-\ln x_i^*-1)+2x_i^* \right) \\
            &\le C_1C_2LrW+C_1C_2^2Lrn=C_3W+C_4,
        \end{split}
    \end{equation}
    where $C_3=C_1C_2Lr,~C_4=2C_1C_2^2Lrn>0$ are constants. Applying Gronwall's inequality‌ to (\ref{eq_pf_5}) it yields
    \begin{equation}
        W(t)\le\left( W_0+ \frac{C_4}{C_3}\right)e^{C_3t}-\frac{C_4}{C_3},\quad t \in [0,t_{max}),
    \end{equation}
    which is a contradiction to $\lim_{t \to t_{max}}W(t)=+\infty$. Therefore, we have $t_{max}=+\infty$.

    \textit{Step 3.} We prove the result of this theorem. Since for any $u^* \in \mathbb{U}$, the MAS is complex balanced, and thus the function (\ref{Lyafun_deficiency_zero}) satisfies all the conditions in \cref{lem_ISS_Lya}. By \cref{thm_ISS_lya}, the conclusion is true.
\end{proof}

\begin{remark}
\cref{coro_cb_mas} imposes no restriction on the deficiency or the number of linkage classes. It can thus be viewed as a generalization of the result in \cite{chaves2005input}.
\end{remark}

\begin{example}\label{ex_ISS_cb}
    Consider the p53 signaling and dimerization network \cite{loriaux2013protein,nicholls2002biogenesis}
    \begin{equation*}
        \begin{tikzpicture}
            \node (a) at (0,0) {$\varnothing$};
            \node (b) at (2,0) {$X_1$};
            \node (c) at (2,2) {$X_1+X_2$};
            \node (d) at (0,2) {$X_2$};
            \draw[->] (a) -- (b) node[midway, above] {$u_1$};
            \draw[->] (b) -- (c) node[midway, right] {$u_2$};
            \draw[->] (c) -- (d) node[midway, above] {$u_3$};
            \draw[->] (d) -- (a) node[midway, left] {$u_4$};
            \node (x) at (2.7,-0.1) {,};
            \node (e) at (5,1) {$2X_1 \ce{<=>[$u_5$][$u_6$]} X_3$~,};
        \end{tikzpicture}
    \end{equation*}
    where $X_1,X_2,X_3$ represent p53 tumor suppressor, its negative regulator Mdm2, and its dimerization product, respectively. The first four reactions (left hand side) represent the negative feedback regulation process of the p53–Mdm2 system, while the last two reactions (right hand side) represent the dimerization process of p53. The dynamics of this time-varying MAS is
    \begin{equation}\label{eq_ex_CB}
        \left ( \begin{array}{c}
         \dot{x}_1 \\ \dot{x}_2 \\ \dot{x}_3
    \end{array} \right) =
     \left ( \begin{array}{c}
         u_1-u_3x_1x_2-2u_5x_1^2+2u_6x_3 \\
         u_2x_1-u_4x_2 \\
         u_5x_1^2-u_6x_3
    \end{array} \right).
    \end{equation}
    Clearly, the CRN is weakly reversible, and has 6 complexes, 2 linkage classes and $\dim \mathscr{S}=3$, which means its deficiency $\delta=6-2-3=1$. Moreover, the MAS is complex balanced if and only if there exists a $x^* \in \mathbb{R}^3_{>0}$ such that
    \begin{equation*}
        \begin{cases}
            u_2x_1^*=u_1 \\
            u_3x_1^*x_2^*=u_2x_1^* \\
            u_4x_2^*=u_3x_1^*x_2^* \\
            u_5(x_1^*)^2=u_6x_3^*
        \end{cases} \Longleftrightarrow ~
        u_1u_3=u_2u_4,~u_5,u_6>0,
    \end{equation*}
    i.e., the toric locus $\mathcal{T}=\{u \in \mathbb{R}^6_{>0}: u_1u_3=u_2u_4,~u_5,u_6>0\}$. By setting $u^*=(2,1,2,4,1,4)^{\top} \in \mathcal{T}$, we get $x^*=(2,0.5,1)^{\top}$. Then according to \cref{coro_cb_mas}, for any compact set $\mathbb{U} \subseteq \mathcal{T}$ containing $u^*$, this MAS, with input-value set $\mathbb{U}$, is semiglobal ISS with respect to $(x^*,u^*)$. To observe the ISS behavior intuitively, the readers may find the corresponding numerical simulations at the end of this paper. 
\end{example}



\section{ISS for non-weakly reversible CRNs}\label{section_4}
In this section, we further extend the analysis of ISS to non-weakly reversible CRNs. 

In general, it is difficult to capture the ISS property of a MAS only based on the dynamics. The above analysis indicates that the structure information of network is quite crucial in catching ISS. To fully utilize this point, we try to make a bridge between the dynamics of the concerned MAS and that of weakly reversible MAS by leveraging the notion of linear conjugacy.

\begin{definition}[Linear conjugacy \cite{johnston2011linear}]
    Two MASs $(\mathcal{N},u)$ and $(\tilde{\mathcal{N}},\tilde{u})$ are linearly conjugate if there exists a linear, bijective mapping $h:\mathbb{R}^n_{>0} \to \mathbb{R}^n_{>0}$ such that $h\left(\Phi(x_0,t)\right)=\tilde{\Phi}\left(\tilde{x}_0,t\right)$ for all $x_0 \in \mathbb{R}^n_{>0}$ and $\tilde{x}_0=h(x_0)$, where $\Phi(x_0,t)$ and $\tilde{\Phi}(\tilde{x}_0,t)$ denote the system orbits of the two MASs with initial values $x_0$ and $\tilde{x}_0$, respectively.
\end{definition}
As discussed in \cite{johnston2011linear}, linear conjugacy also means that there exists a positive, diagonal matrix $D=\mathrm{diag}(d_1,...,d_n)$ to link the dynamics of two MASs through $D\dot{x}=\dot{\tilde{x}}$, i.e.,
\begin{equation}\label{eq_def_LC_MAS}
    D\sum_{j=1}^r u_jx^{v_{\cdot j}}\left(v_{\cdot j}^{\prime}-v_{\cdot j} \right)=\sum_{j=1}^{\tilde{r}}\left(\tilde{u}_j\prod_{i=1}^nd_i^{\tilde{v}_{ij}}\right)x^{\tilde{v}_{\cdot j}}\left(\tilde{v}_{\cdot j}^{\prime}-\tilde{v}_{\cdot j} \right), \quad \forall~x \in \mathbb{R}^n_{\ge 0}.
\end{equation}
Specially, when $D$ is the identity matrix, $(\mathcal{N},u)$ and $(\tilde{\mathcal{N}},\tilde{u})$ are said to be \textit{dynamically equivalent} \cite{deshpande2023source}.

Utilizing the equivalent expression of dynamics (see (\ref{general dynamics_complex})) and denoting the set of reactant complexes by $\mathcal{C}_{react}=\{v_{\cdot 1},...,v_{\cdot r}\}\cup \{\tilde{v}_{\cdot 1},...,\tilde{v}_{\cdot \tilde{r}}\}$, we can rewrite (\ref{eq_def_LC_MAS}) as
\begin{equation*}
    D\sum_{z \in \mathcal{C}_{react}}x^z\sum_{\{j|v_{\cdot j}=z\}}u_j\left(v_{\cdot j}^\prime-v_{\cdot j}\right)=\sum_{z \in \mathcal{C}_{react}}x^z\sum_{\{j|\tilde{v}_{\cdot j}=z\}}\left(\tilde{u}_j\prod_{i=1}^nd_i^{z_{i}}\right) \left(\tilde{v}_{\cdot j}^{\prime}-\tilde{v}_{\cdot j} \right).
\end{equation*}
Note that $z$ is any complex in $\mathcal{C}_{react}$, so the above equation is equivalent to 
\begin{equation}\label{eq_def_LC}
    D\sum_{\{j|v_{\cdot j}=z\}}u_j\left(v_{\cdot j}^\prime-v_{\cdot j}\right)=\sum_{\{j|\tilde{v}_{\cdot j}=z\}}\left(\tilde{u}_j\prod_{i=1}^nd_i^{z_{i}}\right) \left(\tilde{v}_{\cdot j}^{\prime}-\tilde{v}_{\cdot j} \right), \quad \forall~z \in \mathcal{C}_{react}.
\end{equation}
Then we present the main theorem of this subsection, which characterizes the input-value set $\mathbb{U}$ using the concept of linear conjugacy.

\begin{theorem}\label{thm_ISS_LC}
For a time-varying MAS $(\mathcal{S},\mathcal{C},\mathcal{R},u)$ of (\ref{eq_time_varying_CRN}), suppose there is another time-varying MAS $(\mathcal{S},\tilde{\mathcal{C}},\tilde{\mathcal{R}},\tilde{u})$, which is weakly reversible and has toric locus $\mathcal{T}(\tilde{\mathcal{N}})$ with $\tilde{\mathcal{N}}=(\mathcal{S},\tilde{\mathcal{C}},\tilde{\mathcal{R}})$, to be linear conjugate to it, i.e., there exists a positive, diagonal matrix $D=\mathrm{diag}(d_1,...,d_n)$ $(d_i>0)$ to support (\ref{eq_def_LC}). Further, suppose that $\tilde{u}$ depends smoothly on $u \in \mathcal{U}$ with
    \begin{equation}\label{eq_U}
            \mathcal{U}=\{u \in \mathbb{R}^r_{>0}: ~ \exists~ \tilde{u} \in \mathcal{T}(\tilde{\mathcal{N}}) ~such~that~\forall~ z \in \mathcal{C}_{react} ~it ~holds~ (\ref{eq_def_LC})\}.
      \end{equation}
Then for any compact set $\mathbb{U} \subseteq \mathcal{U}$ and any $u^* \in \mathbb{U}$, the system (\ref{eq_time_varying_CRN}), with input-value set $\mathbb{U}$, is semiglobal ISS with respect to $(x^*,u^*)$, where $x^*$ satisfies $f(x^*;u^*)=0$.

\end{theorem}

To prove \cref{thm_ISS_LC}, we introduce two lemmas with the proofs given in the Appendix.

\begin{lemma}\label{lem_LCthm_1}
    Suppose all the conditions in \cref{thm_ISS_LC} are satisfied. Then for any $u \in \mathcal{U}$ given in (\ref{eq_U}), the MAS $(\mathcal{S},\mathcal{C},\mathcal{R},u)$ has a unique equilibrium $x^{*}=x^*(u)$ in each positive stoichiometric compatibility class. Moreover, $x^*=x^*(u)$ depends smoothly on the parameter values $u \in \mathcal{U}$.
\end{lemma}

\begin{lemma}\label{lem_LCthm_2}
    Suppose all the conditions in \cref{thm_ISS_LC} are satisfied. Then for any compact set $\mathbb{U} \subseteq \mathcal{U}$ given in (\ref{eq_U}), there exists a function $V_0(x,u)$ satisfying condition (1)-(3) in \cref{lem_ISS_Lya}.
\end{lemma}

\begin{proof}[Proof of \cref{thm_ISS_LC}]
    From \cref{lem_LCthm_1,lem_LCthm_2}, we know that all the conditions in \cref{lem_ISS_Lya} are satisfied. Then by \cref{thm_ISS_lya}, the conclusion is true.
\end{proof}

\begin{example}\label{ex_LC}
 We use the MAS $(\mathcal{N},u)$ discussed in \cite{johnston2011linear} to exhibit ISS, which takes 
    \begin{equation*}
        X_1+2X_2 \overset{u_1}{\longrightarrow} X_1+3X_2 \overset{u_2}{\longrightarrow}X_1+X_2 \overset{u_3}{\longrightarrow}3X_1,~2X_1 \overset{u_4}{\longrightarrow} X_2
    \end{equation*}
  with dynamics
    \begin{equation}\label{eq_ex_LC}
        \left ( \begin{array}{c}
         \dot{x}_1 \\ \dot{x}_2 
    \end{array} \right) =
     \left ( \begin{array}{c}
         2u_3x_1x_2-2u_4x_1^2 \\ u_1x_1x_2^2-2u_2x_1x_2^3-u_3x_1x_2+u_4x_1^2 
    \end{array} \right).
    \end{equation}
This MAS was said to be linear conjugate to the following MAS $(\tilde{\mathcal{N}},\tilde{u})$
    \begin{equation*}
       X_1+2X_2 \ce{<=>[\tilde{u}_1][\tilde{u}_2]} X_1+3X_2,~ X_1+X_2 \ce{<=>[\tilde{u}_3][\tilde{u}_4]} 2X_1
    \end{equation*}
with dynamics
    \begin{equation}\label{eq_ex_LC2}
        \left ( \begin{array}{c}
         \dot{x}_1 \\ \dot{x}_2 
    \end{array} \right) =
     \left ( \begin{array}{c}
         \tilde{u}_3x_1x_2-\tilde{u}_4x_1^2 \\ \tilde{u}_1x_1x_2^2-\tilde{u}_2x_1x_2^3-\tilde{u}_3x_1x_2+\tilde{u}_4x_1^2 
    \end{array} \right),
    \end{equation}
linked by $D=\mathrm{diag}(1,2)$. Note that the second MAS is weakly reversible, and has deficiency $\delta=4-2-2=0$, so we have $\mathcal{T}(\tilde{\mathcal{N}})=\mathbb{R}^4_{>0}$. Then (\ref{eq_def_LC}) follows
    \begin{equation}\label{eq_ex_LC3}
        \begin{cases}
             \left ( \begin{array}{cc}
         1 &  \\
           & 2 \\
    \end{array} \right)u_1 \left ( \begin{array}{c}
         0  \\
         1  \\
    \end{array} \right) = \tilde{u}_1 \cdot 4 \left ( \begin{array}{c}
         0  \\
         1  \\
    \end{array} \right), \\
            \left ( \begin{array}{cc}
         1 &  \\
           & 2 \\
    \end{array} \right)u_2 \left ( \begin{array}{c}
         0  \\
         -2  \\
    \end{array} \right) =  \tilde{u}_2 \cdot 8 \left ( \begin{array}{c}
         0  \\
         -1  \\
    \end{array} \right), \\
    \left ( \begin{array}{cc}
         1 &  \\
           & 2 \\
    \end{array} \right)u_3 \left ( \begin{array}{c}
         2  \\
         -1  \\
    \end{array} \right) = \tilde{u}_3 \cdot 2 \left ( \begin{array}{c}
         1  \\
         -1  \\
    \end{array} \right), \\
    \left ( \begin{array}{cc}
         1 &  \\
           & 2 \\
    \end{array} \right)u_4 \left ( \begin{array}{c}
         -2  \\
         1  \\
    \end{array} \right) =  \tilde{u}_4 \left ( \begin{array}{c}
         -1  \\
         1  \\
    \end{array} \right).
        \end{cases}  \Longleftrightarrow ~
        \begin{cases}
            u_1=2 \tilde{u}_1, \\
            u_2=2 \tilde{u}_2, \\
            u_3= \tilde{u}_3, \\
            u_4=\frac{1}{2} \tilde{u}_4.
        \end{cases}
    \end{equation}
    Therefore, according to (\ref{eq_U}) we have
    $$\mathcal{U}=\left\{u \in \mathbb{R}^4_{>0}:~ (u_1,u_2,u_3,u_4)=\left(2\tilde{u}_1,2\tilde{u}_2,\tilde{u}_3,\frac{1}{2}\tilde{u}_4\right), \tilde{u} \in \mathbb{R}^4_{>0}\right\}=\mathbb{R}^4_{>0}.$$
    By setting $u^*=(2,1,3,1)^{\top} \in \mathcal{U}$, we get $x^*=(1,3)^{\top}$. Then according to \cref{thm_ISS_LC}, for any compact set $\mathbb{U} \subseteq \mathbb{R}^4_{>0}$ containing $u^*$, the MAS (\ref{eq_ex_LC}), with input-value set $\mathbb{U}$, is semiglobal ISS with respect to $(x^*,u^*)$.
\end{example}

As shown in \textit{Example} \ref{ex_LC}, once the complex balanced MAS $(\tilde{\mathcal{N}},\tilde{u})$ and the linking matrix $D$ are specified, determining the set $\mathcal{U}$ reduces to solving a system of linear equations (like (\ref{eq_ex_LC3})). However, this system of linear equations may have no solution. That is, the set $\mathcal{U}$ in \cref{thm_ISS_LC} may be empty. To avoid this possibility, it is necessary to investigate the conditions under which $\mathcal{U}$ is nonempty. For this purpose, similar to the methods in \cite{johnston2011linear}, we define two sets
\begin{equation}
    C_{\mathcal{R},\mathbb{K}}(z)=\left\{\sum_{\{j|v_{\cdot j}=z\}}u_j\left(v_{\cdot j}^{\prime}-v_{\cdot j}\right): u \in \mathbb{K}\right\}
\end{equation}
and
\begin{equation}
    C_{\tilde{\mathcal{R}},\tilde{\mathbb{K}}}(z)=\left\{\sum_{\{j|\tilde{v}_{\cdot j}=z\}}\left(\tilde{u}_j\prod_{i=1}^nd_i^{z_i}\right)\left(\tilde{v}_{\cdot j}^{\prime}-\tilde{v}_{\cdot j}\right): \tilde{u} \in \tilde{\mathbb{K}} \right\},
\end{equation}
where $z \in \mathcal{C}_{react},~\mathbb{K} \subseteq \mathbb{R}^r_{>0}$ and $\tilde{\mathbb{K}}\subseteq \mathbb{R}^{\tilde{r}}_{>0}$.

\begin{proposition}
    The set $\mathcal{U}$ defined by (\ref{eq_U}) is not empty if and only if for every $z \in \mathcal{C}_{react}$ we have $D~C_{\mathcal{R},\mathbb{R}^r_{>0}}(z) \cap C_{\tilde{\mathcal{R}},\mathcal{T}(\tilde{\mathcal{N}})}(z) \ne \varnothing$, where $D~C_{\mathcal{R},\mathbb{R}^r_{>0}}(z)=\{D \xi : \xi \in C_{\mathcal{R},\mathbb{R}^r_{>0}}(z)\}$.
\end{proposition}

\begin{proof}
    The proof is trivial, since $D~C_{\mathcal{R},\mathbb{R}^r_{>0}}(z) \cap C_{\tilde{\mathcal{R}},\mathcal{T}(\tilde{\mathcal{N}})}(z) \ne \varnothing,~\forall~z \in \mathcal{C}_{react}$ if and only if $\exists~u \in \mathbb{R}^r_{>0},~\tilde{u} \in \mathcal{T}(\tilde{\mathcal{N}})$ such that (\ref{eq_def_LC}) holds.
\end{proof}

Another point worth noting is the condition under which $\mathcal{U}=\mathbb{R}^r_{>0}$, which implies that $\mathbb{U}$ can be any compact subset of positive orthant $\mathbb{R}^r_{>0}$.

\begin{proposition}\label{prop_LC_U=R}
    The set defined by (\ref{eq_U}) satisfies $\mathcal{U} = \mathbb{R}^r_{>0}$ if and only if for every $z \in \mathcal{C}_{react}$ we have $D~C_{\mathcal{R},\mathbb{R}^r_{>0}}(z) \subseteq C_{\tilde{\mathcal{R}},\mathcal{T}(\tilde{\mathcal{N}})}(z)$.
\end{proposition}

\begin{proof}
    ``\textit{if}'': $\forall~u \in \mathbb{R}^r_{>0}$, by the given condition we know that $\exists~\tilde{u} \in \mathcal{T}(\tilde{\mathcal{N}})$ such that for every $z \in \mathcal{C}_{react}$, it holds (\ref{eq_def_LC}), which means $u \in \mathcal{U}$. Since $u$ is arbitrary, it follows that $\mathbb{R}^r_{>0} \subseteq \mathcal{U}$. It is obvious that $\mathcal{U} \subseteq \mathbb{R}^r_{>0}$, we thus have $\mathcal{U}=\mathbb{R}^r_{>0}$.

    ``\textit{only if}'': For every $z \in \mathcal{C}_{react}$, $\forall~\xi \in D~C_{\mathcal{R},\mathbb{R}^r_{>0}}(z)$, $\exists~u\in\mathbb{R}^r_{>0}$ such that $\xi=D\sum_{\{j|v_{\cdot j}=z\}}u_j (v_{\cdot j}^{\prime}-v_{\cdot j})$. Since $\mathcal{U}=\mathbb{R}^r_{>0}$, $\exists~\tilde{u} \in \mathcal{T}(\tilde{\mathcal{N}})$ such that
    $$\xi=D\sum_{\{j|v_{\cdot j}=z\}}u_j(v_{\cdot j}^{\prime}-v_{\cdot j})=\sum_{\{j|\tilde{v}_{\cdot j}=z\}}\left(\tilde{u}_j\prod_{i=1}^nd_i^{z_{i}}\right) \left(\tilde{v}_{\cdot j}^{\prime}-\tilde{v}_{\cdot j} \right),$$
    which means $\xi \in C_{\tilde{\mathcal{R}},\mathcal{T}(\tilde{\mathcal{N}})}(z)$. Since $\xi$ is arbitrary, we have $ D~C_{\mathcal{R},\mathbb{R}^r_{>0}}(z) \subseteq C_{\tilde{\mathcal{R}},\mathcal{T}(\tilde{\mathcal{N}})}(z)$.
\end{proof}

Revisit \textit{Example} \ref{ex_LC}, for $z=(1,2)^{\top} \in \mathcal{C}_{react}$, we can get
\begin{equation*}
    \begin{split}
        D~C_{\mathcal{R},\mathbb{R}^r_{>0}}(z)&=\left\{u_1
    \left ( \begin{array}{c}
         0  \\
         2  \\
    \end{array} \right):u \in \mathbb{R}^4_{>0}
    \right\} \\
    C_{\tilde{\mathcal{R}},\mathcal{T}(\tilde{\mathcal{N}})}(z)&=\left\{\tilde{u}_1
    \left ( \begin{array}{c}
         0  \\
         4  \\
    \end{array} \right):\tilde{u} \in \mathcal{T}(\tilde{\mathcal{N})}=\mathbb{R}^4_{>0}
    \right\}
    \end{split}
\end{equation*}
and thus $ D~C_{\mathcal{R},\mathbb{R}^r_{>0}}(z) \subseteq C_{\tilde{\mathcal{R}},\mathcal{T}(\tilde{\mathcal{N}})}(z)$. Similarly, one can verify that the above inclusion relation holds for any $z \in \mathcal{C}_{react}$. By \cref{prop_LC_U=R} we have $\mathcal{U}=\mathbb{R}^4_{>0}$.

\begin{proposition}\label{prop_LC_U=R_2}
    The set defined by (\ref{eq_U}) satisfies $\mathcal{U} = \mathbb{R}^r_{>0}$ if the following conditions hold:
    \begin{enumerate}
        \item[(1)] $\tilde{\mathcal{N}}$ is weakly reversible and has zero deficiency;
        \item[(2)] $r=\tilde{r}$, and $v_{\cdot i}\ne v_{\cdot j},~\forall~1\le i< j\le r$;
        \item[(3)] $\exists~u,\tilde{u} \in \mathbb{R}^r_{>0}$ such that (\ref{eq_def_LC}) holds.
    \end{enumerate}
\end{proposition}

\begin{proof}
    We complete the proof in three steps.

    \textit{Step 1.} Denote $\mathcal{C}_1=\{v_{\cdot 1},...,v_{\cdot r}\},~\tilde{\mathcal{C}}_1=\{\tilde{v}_{\cdot 1},...,\tilde{v}_{\cdot \tilde{r}}\}$, then we prove that $\mathcal{C}_1 \subseteq \tilde{\mathcal{C}}_1$. This proof is proceeded by contradiction. Suppose $\exists~1\le j\le r$ such that $v_{\cdot j} \notin \tilde{\mathcal{C}}_1$. According to condition (3) and $v_{\cdot i}\ne v_{\cdot j},~\forall~i\ne j$, we have $D u_j(v_{\cdot j}^{\prime}-v_{\cdot j})=\sum_{\{l|\tilde{v}_{\cdot l}=v_{\cdot j}\}}\tilde{u}_l\left(\prod_{i=1}^nd_i^{v_{ij}}\right)(\tilde{v}_{\cdot l}^{\prime}-\tilde{v}_{\cdot l})=0.$ Since $D$ is invertible and $v_{\cdot j}\ne v_{\cdot j}^{\prime}$, we can get $u_j=0$. This is in contradiction to $u \in \mathbb{R}^r_{>0}$. Therefore, we have $v_{\cdot j} \in \tilde{\mathcal{C}}_1,~\forall~1\le j\le r$, i.e., $\mathcal{C}_1 \subseteq \tilde{\mathcal{C}}_1$.

    \textit{Step 2.} We prove that $\mathcal{C}_1=\tilde{\mathcal{C}}_1$, and $\tilde{v}_{\cdot i}\ne \tilde{v}_{\cdot j},~\forall~1\le i<j\le \tilde{r}$. The proof is trivial, since $\mathcal{C}_1 \subseteq \tilde{\mathcal{C}}_1$ and $r=\tilde{r}$.

    \textit{Step 3.} We prove that $\mathcal{U}=\mathbb{R}^r_{>0}$. According to \textit{Step 2} we have $\mathcal{C}_{react}=\mathcal{C}_1=\tilde{\mathcal{C}}_1$. For simplicity and without loss of generality, we set $v_{\cdot j}=\tilde{v}_{\cdot j},~1\le j\le r$. For any $\mu \in \mathbb{R}^r_{>0}$, denote $c=(\frac{\mu_1}{u_1},...,\frac{\mu_r}{u_r})^{\top} \in \mathbb{R}^r_{>0}$, then $\forall~1\le j\le r$ we have
    \begin{equation*}
    \begin{split}
        &D\mu_j\left(v_{\cdot j}^{\prime}-v_{\cdot j}\right)=c_jDu_j\left(v_{\cdot j}^{\prime}-v_{\cdot j}\right) \\
        &= c_j \tilde{u}_j\left(\prod_{i=1}^nd_i^{v_{ij}}\right)(\tilde{v}_{\cdot j}^{\prime}-\tilde{v}_{\cdot j}) = \tilde{\mu}_j\left(\prod_{i=1}^nd_i^{v_{ij}}\right)(\tilde{v}_{\cdot j}^{\prime}-\tilde{v}_{\cdot j})
    \end{split}
    \end{equation*}
    where $\tilde{\mu}=(c_1\tilde{u}_1,...,c_r\tilde{u}_r)^{\top} \in \mathbb{R}^r_{>0}$. By condition (1) it follows $\mathbb{R}^r_{>0}=\mathcal{T}(\tilde{\mathcal{N}})$. Therefore, we have $\mu \in \mathcal{U}$. Since $\mu$ is arbitrary, it follows that $\mathbb{R}^r_{>0} \subseteq \mathcal{U}$. It is obvious that $\mathcal{U} \subseteq \mathbb{R}^r_{>0}$, we thus have $\mathcal{U}=\mathbb{R}^r_{>0}$.
\end{proof}

Recall that $\tilde{\mathcal{N}}$ in \textit{Example} \ref{ex_LC} is weakly reversible and has zero deficiency, and $r=\tilde{r}=4,~v_{\cdot i}\ne v_{\cdot j},~1\le i<j\le 4$. In addition, let $u=(2,2,1,1)^{\top},~\tilde{u}=(1,1,1,2)^{\top}$, and then (\ref{eq_def_LC}) holds. Then by \cref{prop_LC_U=R_2} we have $\mathcal{U}=\mathbb{R}^4_{>0}$.

It should be noted that all the above propositions are derived under the condition that $D$ and $\tilde{\mathcal{N}}$ are given. When only $\mathcal{N}$ is provided, finding a weakly reversible network $\tilde{\mathcal{N}}$ that is linearly conjugate to it under $D$ is not an easy task. This task is called ``linearly conjugate weakly reversible realizations'', and in the case $D=I$, it is also known as ``weakly reversible realizations''. The main approach to solve these problems is to reformulate it as a mixed-integer linear programming problem \cite{johnston2013computing,craciun2023algorithm}.

The network in \textit{Example} \ref{ex_LC} is an artificially constructed abstract network and does not necessarily have a direct biological interpretation. The following example establish the ISS of a biological network with practical significance.
\begin{example}
    Consider the p21-activated kinase 1 (PAK1) network
\begin{equation*}
    2X_1\overset{u_1}{\longrightarrow}X_1+X_2,\quad X_1 \overset{u_2}{\longleftarrow}X_2 \ce{<=>[$u_3$][$u_4$]} X_3,
\end{equation*}
where $X_1,X_2,X_3$ denote unphosphorylated PAK1, monophosphorylated PAK1, and diphosphorylated PAK1, respectively. The network plays a crucial role in the regulation of cell motility and morphology \cite{hong2021derivation}. The dynamics of this time-varying MAS is
\begin{equation}\label{eq_supp_ex}
        \left ( \begin{array}{c}
         \dot{x}_1 \\ \dot{x}_2 \\ \dot{x}_3
    \end{array} \right) =
     \left ( \begin{array}{c}
         -u_1x_1^2+u_2x_2\\
         u_1x_1^2-u_2x_2-u_3x_2+u_4x_3 \\
         u_3x_2-u_4x_3
    \end{array} \right).
\end{equation}
This MAS was said to be linear conjugate to the following MAS $(\tilde{\mathcal{N}},\tilde{u})$
\begin{equation*}
    2X_1 \ce{<=>[$\tilde{u}_1$][$\tilde{u}_2$]} X_2 \ce{<=>[$\tilde{u}_3$][$\tilde{u}_4$]} X_3
\end{equation*}
with dynamics
\begin{equation}
        \left ( \begin{array}{c}
         \dot{x}_1 \\ \dot{x}_2 \\ \dot{x}_3
    \end{array} \right) =
     \left ( \begin{array}{c}
         -2\tilde{u}_1x_1^2+2\tilde{u}_2x_2\\
         \tilde{u}_1x_1^2-\tilde{u}_2x_2-\tilde{u}_3x_2+\tilde{u}_4x_3 \\
         \tilde{u}_3x_2-\tilde{u}_4x_3
    \end{array} \right),
\end{equation}
linked by $D=\mathrm{diag}(2,1,1)$. Clearly, the second MAS is weakly reversible, and has deficiency $\delta=3-1-2=0$, so it follows $\mathcal{T}(\tilde{\mathcal{N}})=\mathbb{R}^4_{>0}$. Further, for $z=(2,0,0)$ we have
\begin{equation*}
\begin{split}
    D~C_{\mathcal{R},\mathbb{R}^r_{>0}}(z)&=\left\{u_1(-2,1,0)^{\top}:u \in \mathbb{R}^4_{>0} \right\} \\
    C_{\tilde{\mathcal{R}},\mathcal{T}(\tilde{\mathcal{N}})}(z)&=\left\{4\tilde{u}_1(-2,1,0)^{\top}:\tilde{u} \in \mathcal{T}(\tilde{\mathcal{N}})=\mathbb{R}^4_{>0} \right\}
\end{split}
\end{equation*}
and thus $D~C_{\mathcal{R},\mathbb{R}^r_{>0}}(z) \subseteq C_{\tilde{\mathcal{R}},\mathcal{T}(\tilde{\mathcal{N}})}(z)$. Similarly, one can verify that the above inclusion relation holds for any $z \in \mathcal{C}_{react}$. By \cref{prop_LC_U=R} we know that the set $\mathcal{U}$ defined in (\ref{eq_U}) satisfies $\mathcal{U}=\mathbb{R}^4_{>0}$. In addition, by setting $u^*=(1,4,2,5)^{\top}$, we obtain an equilibrium $x^*=(2,1,0.4)^{\top}$. Then according to \cref{thm_ISS_LC}, for any compact set $\mathbb{U} \subseteq \mathbb{R}^4_{>0}$ containing $u^*$, the MAS (\ref{eq_supp_ex}), with input-value set $\mathbb{U}$, is semiglobal ISS with respect to $(x^*,u^*)$. Also, the readers may refer to the end of this paper to find numerical simulations. 



\end{example}


\section{Application}\label{section_5}
Recently, CRNs have been widely used as a mathematical framework for studying molecular computations \cite{anderson2025chemical,fan2025automatic}. In this framework, the initial concentrations of certain species (i.e., $x_i(0)$) are treated as inputs, while the limiting steady state concentrations of other species (i.e., $\lim_{t \to \infty}x_j(t)$), are interpreted as outputs. By designing the resulting input-output relation to match a prescribed target function, one can theoretically implement the molecular computation of that function. Among these studies, an important question is: if two MASs $(\mathcal{N}_1,\kappa_1)$ and $(\mathcal{N}_2,\kappa_2)$ serve to compute functions $u=\sigma_1(w)$ and $x=\sigma_2(u)$, respectively, can their coupling compute the corresponding composite function $x=\sigma_2\circ \sigma_1(w)$? If the answer is affirmative, then $(\mathcal{N}_1,\kappa_1)$ and $(\mathcal{N}_2,\kappa_2)$ are ``dynamically composable'' \cite{jiang2026structure}.

Using the assumptions and techniques developed in \cite{jiang2026structure}, we can view the coupling dynamics as a time-varying system (\ref{eq_time_varying_CRN}), where the input $u(t)$ is generated by the dynamics of $(\mathcal{N}_1,\kappa_1)$, and the key issue in establishing dynamical composability reduces to \textbf{Question 2} in Section \ref{section_2}. The following theorem provides a partial answer to this question.

\begin{theorem}\label{thm_globally_stable}
    Consider a MAS $(\mathcal{S},\mathcal{C},\mathcal{R},u)$ of (\ref{eq_time_varying_CRN}) starting from $x(0)=x_0$, and suppose $x^*$ is its unique equilibrium in $(x_0+\mathscr{S})\cap \mathbb{R}^n_{>0}$ under $u(t) \equiv u^*\in \mathbb{R}^r_{>0}$. Then the limiting state of (\ref{eq_time_varying_CRN}) satisfies $\lim_{t \to \infty}x(t)=x^*$ if 
    \begin{enumerate}
        \item[(1)] $u(t) \in \mathbb{U},~\forall~t\ge 0$, where $\mathbb{U} \subseteq \mathbb{R}^r_{>0}$ is a compact set, and $\lim_{t \to \infty}u(t)=u^*$; 
        \item[(2)] the MAS of (\ref{eq_time_varying_CRN}), with input-value set $\mathbb{U}$, is semiglobal ISS with respect to $(x^*,u^*)$;
        \item[(3)] the MAS of (\ref{eq_time_varying_CRN}) is permanent, that is, there exists $\varepsilon>0$ such that for any initial value $x_0 \in \mathbb{R}^n_{>0}$, the solution of (\ref{eq_time_varying_CRN}) satisfies $\liminf_{t \to \infty} x_i(t)>\varepsilon$ and $\limsup_{t \to \infty} x_i(t) < 1/\varepsilon$ for all $i=1,2,...,n$.
    \end{enumerate}
\end{theorem}

\begin{proof}
  By condition (3), there exists a compact set $F \subseteq \mathbb{R}^n_{>0}$ such that $x(t) \in F,~\forall~t\ge 0$. Using condition (2) and treating $s\ge 0$ as the initial time, there exist a class $\mathcal{KL}$ function $\beta$ and a class $\mathcal{K}_{\infty}$ function $\gamma$ such that
    \begin{equation}\label{eq_pf3_1}
        |x(t)-x^*|\le \beta(|x(s)-x^*|,t-s)+\gamma\left(\mathrm{ess.sup}\{|u(\tau)-u^*|:\tau \ge s\}\right),
    \end{equation}
    where $0\le s\le t$. Substituting $s=\frac{t}{2}$ into (\ref{eq_pf3_1}) yields
    \begin{equation}\label{eq_pf3_2}
        |x(t)-x^*|\le \beta\left(|x(\frac{t}{2})-x^*|,\frac{t}{2}\right)+\gamma\left(\mathrm{ess.sup}\{|u(\tau)-u^*|:\tau \ge \frac{t}{2}\}\right).
    \end{equation}
    By $u(t) \to u^*$ we obtain $\gamma\left(\mathrm{ess.sup}\{|u(\tau)-u^*|:\tau \ge \frac{t}{2}\}\right) \to 0$ as $t \to \infty$. To estimate $|x(\frac{t}{2})-x^*|$, apply (\ref{eq_pf3_1}) with $s=0$ and $t$ replaced by $\frac{t}{2}$, then we have
    \begin{equation}\label{eq_pf3_3}
        |x(\frac{t}{2})-x^*|\le \beta\left(|x_0-x^*|,\frac{t}{2}\right)+\gamma\left(\mathrm{ess.sup}\{|u(\tau)-u^*|:\tau \ge 0\}\right).
    \end{equation}
    Since $u(t) \in \mathbb{U}$ is bounded, there exists a constant $M>0$ such that 
    \begin{equation}\label{eq_pf3_4}
        \gamma\left(\mathrm{ess.sup}\{|u(\tau)-u^*|:\tau \ge 0\}\right)\le M.
    \end{equation}
    In addition, since $\beta$ is a class $\mathcal{KL}$ function, for sufficiently large $t$, it holds
    \begin{equation}\label{eq_pf3_5}
        \beta \left(|x_0-x^*|,\frac{t}{2}\right)\le M.
    \end{equation}
    Substituting \cref{eq_pf3_3,eq_pf3_4,eq_pf3_5} into (\ref{eq_pf3_2}), we obtain
    \begin{equation}
        \begin{split}
            |x(t)-x^*|& \le \beta\left(2M,\frac{t}{2}\right)+\gamma\left(\mathrm{ess.sup}\{|u(\tau)-u^*|:\tau \ge \frac{t}{2}\}\right) \\
            & \to 0 \quad (t \to \infty)
        \end{split}
    \end{equation}
    that is, $\lim_{t \to \infty}x(t)=x^*$.
\end{proof}

\cref{thm_globally_stable} provides a sufficient condition for the global asymptotic stability of the time-varying system. The permanence condition involved in the assumptions has been widely used in the analysis of global asymptotic stability for MASs \cite{craciun2013persistence,gopalkrishnan2014geometric}.

\begin{example}\label{ex_global_stable}
    Revisit the CRNs in \textit{Example} \ref{ex_LC}. Consider an input $u:\mathbb{R}_{\ge 0} \to \mathbb{R}^r_{>0}$ that satisfies $\lim_{t \to \infty}u(t)=u^*$. Then $u(t)$ is bounded, and there exists a compact set $\mathbb{U} \subseteq \mathbb{R}^r_{>0}$ such that $u(t) \in \mathbb{U},~\forall~t\ge 0$. That is, the condition (1) in \cref{thm_globally_stable} holds. According to the discussion in \textit{Example} \ref{ex_LC}, the condition (2) also holds.
    
    To obtain the permanence of time-varying system (\ref{eq_ex_LC}) with initial value $x(0)=x_0$ and $u(t) \in \mathbb{U}$, consider the CRN $\tilde{\mathcal{N}}$ corresponding to time-varying system (\ref{eq_ex_LC2}) with initial value $\tilde{x}(0)=\mathrm{diag}(1,2)\cdot x_0$, where $(\tilde{u}_1,\tilde{u}_2,\tilde{u}_3,\tilde{u}_4)=\left(\frac{1}{2}u_1,\frac{1}{2}u_2,u_3,2u_4\right)$. Then it can be directly verified that $x(t)=\mathrm{diag}(1,1/2)\cdot \tilde{x}(t),~\forall~t\ge 0$, and $\tilde{u}(t) \in \tilde{\mathbb{U}}$, where $\tilde{\mathbb{U}}$ is a compact subset of $\mathbb{R}^4_{>0}$. In addition, since $\tilde{\mathcal{N}}$ is weakly reversible, by \textit{Theorem} 6.4 in \cite{craciun2013persistence} we know that $\tilde{x}(t)$ is permanent. Therefore, $x(t)$ is also permanent. That is, the condition (3) in \cref{thm_globally_stable} holds.

    Since all the conditions in \cref{thm_globally_stable} are satisfied, we have $\lim_{t \to \infty}x(t)=x^*$.
\end{example}

\textit{Example} \ref{ex_global_stable} shows that the time-varying system (\ref{eq_ex_LC}) satisfies $x(t) \to x^*$ if $u(t) \to u^*$. Seeing $u$ and $x$ as the upstream and downstream computational modules in molecular computations, respectively, this property indicates that the coupled system can implement a layer-by-layer computation through parallel chemical reactions. Specifically, consider an upstream computational module 
\begin{equation}\label{cascade_app1}
    U_2 \ce{<=>[$1$][$1$]} U_3, ~U_2 \overset{1}{\longrightarrow}U_2+U_4,~2U_4 \overset{1}{\longrightarrow} \varnothing
\end{equation}
with dynamics to be $(\dot{u}_1,\dot{u}_2,\dot{u}_3,\dot{u}_4)^{\top}=(0,u_3-u_2,u_2-u_3,u_2-2u_4^2)^{\top}$. This module implements the molecular computation of the function $v=\sigma_1(u)=\left(u_1,\frac{u_2+u_3}{2}, \frac{u_2+u_3}{2},\frac{\sqrt{u_2+u_3}}{2}\right)^{\top}$. The downstream computational module
\begin{equation}\label{cascade_app2}
    \begin{split}
        &X_1+2X_2+U_1 \overset{1}{\longrightarrow} X_1+3X_2+U_1,~2X_1+U_4 \overset{1}{\longrightarrow} X_2+U_4, \\
    &X_1+3X_2+U_2 \overset{1}{\longrightarrow}X_1+X_2+U_2,~X_1+X_2+U_3 \overset{1}{\longrightarrow}3X_1+U_3,
    \end{split}
\end{equation}
which has dynamics (\ref{eq_ex_LC}), implements the molecular computation of $x=\sigma_2(v)=\left(\frac{v_1v_3}{2v_2v_4},\frac{v_1}{2v_2}\right)^{\top}$. By \textit{Example} \ref{ex_global_stable} it is concluded that the coupled system below
\begin{equation}\label{coulped_app}
\begin{split}
    &U_2 \ce{<=>[$1$][$1$]} U_3, ~U_2 \overset{1}{\longrightarrow}U_2+U_4,~2U_4 \overset{1}{\longrightarrow} \varnothing,\\ 
    &X_1+2X_2+U_1 \overset{1}{\longrightarrow} X_1+3X_2+U_1,~2X_1+U_4 \overset{1}{\longrightarrow} X_2+U_4, \\
    &X_1+3X_2+U_2 \overset{1}{\longrightarrow}X_1+X_2+U_2,~X_1+X_2+U_3 \overset{1}{\longrightarrow}3X_1+U_3,
\end{split}
\end{equation}
with dynamics to be 
\begin{equation}\label{eq_coupled_system}
    \left ( \begin{array}{c}
    \dot{u}_1 \\ \dot{u}_2 \\ \dot{u}_3 \\ \dot{u}_4 \\
         \dot{x}_1 \\ \dot{x}_2 
    \end{array} \right) =
     \left ( \begin{array}{c}
     0 \\ u_3-u_2 \\ u_2-u_3 \\ u_2-2u_4^2 \\
         2u_3x_1x_2-2u_4x_1^2 \\ u_1x_1x_2^2-2u_2x_1x_2^3-u_3x_1x_2+u_4x_1^2 
    \end{array} \right),
\end{equation}
can implement the molecular computation of the composite function $\sigma_2\circ \sigma_1$, i.e., $(x_1,x_2)^{\top}=\sigma_2(\sigma_1(u))=\left(\frac{u_1}{\sqrt{u_2+u_3}},\frac{u_1}{u_2+u_3}\right)^{\top}$. This suggests the computation by chemical reactions in (\ref{coulped_app}) follows the same result as given by cascading the computational result of reactions in (\ref{cascade_app1}) and that of reactions in (\ref{cascade_app2}). The corresponding numerical simulations are given at the end of this paper.




\section{Conclusions}\label{section_6}
In this paper, we investigate the ISS of CRNs with time-varying reaction-rate inputs, building on the work of Chaves and Sontag. For mass-action kinetics, we first extend the established ISS results for weakly reversible networks with zero deficiency and single linkage class to a broader class of weakly reversible networks. In particular, by restricting the input-value set to a compact subset of the toric locus, we establish semiglobal ISS for weakly reversible networks that have arbitrary deficiency and linkage classes. Then, by exploiting the concept of linear conjugacy, we further extend these ISS results to a class of networks that are not weakly reversible, even though their toric locus are empty. Finally, we apply the above results to molecular computations. We show that, when the reaction-rate input converges to a constant value and the corresponding trajectory is permanent, the semiglobal ISS property guarantees convergence of the state to the equilibrium associated with the limiting input, thereby providing a theoretical foundation for the implementation of parallel molecular computations. A concrete example illustrates the application of the proposed theory to molecular computations.

Several directions deserve further investigation. First, according to the results established in Section \ref{section_4}, the set $\mathcal{U}$ determines the admissible range of the input $u(\cdot)$. Therefore, it is necessary to develop efficient computational methods for determining an admissible, or even the maximal admissible, input set $\mathcal{U}$ for a given CRN. Second, the notion of ISS can be extended to CRNs with time delays, and one may further investigate structural conditions under which delayed networks possess the ISS property. In addition, extending the analysis to other classes of kinetics, like Michaelis–Menten kinetics, would also be of both theoretical and practical significance.


\appendix
\section{Proofs} 

\begin{proof}[\textbf{Proof of \cref{lem_LCthm_1}}]
    We complete the proof in three steps.

    \textit{Step 1.} We prove the existence of the equilibrium in $(x_0+\mathscr{S})\cap \mathbb{R}^n_{>0}$. For any $u \in \mathcal{U}$, we know that there exists a $\tilde{u}$ such that $\forall~v \in \mathcal{C}_{react}$ it holds (\ref{eq_def_LC}) and $(\tilde{\mathcal{N}},\tilde{u})$ is complex balanced, which means there exists an equilibrium in each positive stoichiometric compatibility class. Specifically, suppose $\tilde{x}^* \in (Dx_0+\tilde{\mathscr{S}})\cap \mathbb{R}^n_{>0}$ is the equilibrium of $(\tilde{\mathcal{N}},\tilde{u})$ and thus satisfies $\sum_{j=1}^{\tilde{r}}\tilde{u}_j(\tilde{x}^*)^{\tilde{v}_{\cdot j}}(\tilde{v}_{\cdot j}^{\prime}-\tilde{v}_{\cdot j})=0$. Then we have
    \begin{equation}
        \begin{split}
            &D\sum_{j=1}^ru_j\left(D^{-1}\tilde{x}^*\right)^{v_{\cdot j}}\left(v_{\cdot j}^{\prime}-v_{\cdot j}\right)
            =D\sum_{j=1}^ru_j\left(\tilde{x}^*\right)^{v_{\cdot j}}\left(\prod_{i=1}^nd_i^{-v_{ij}}\right)\left(v_{\cdot j}^{\prime}-v_{\cdot j}\right) \\
            =&\sum_{v \in \mathcal{C}_{react}}(\tilde{x}^*)^{v}\left(\prod_{i=1}^nd_i^{-v_{i}}\right)D\sum_{\{j~:~v_{\cdot j}=v\}}u_j\left(v_{\cdot j}^{\prime}-v_{\cdot j}\right) \\
            =& \sum_{v \in \mathcal{C}_{react}}(\tilde{x}^*)^{v}\left(\prod_{i=1}^nd_i^{-v_{i}}\right)\sum_{\{j~:~\tilde{v}_{\cdot j}=v\}}\left(\tilde{u}_j\prod_{i=1}^nd_i^{v_i}\right)\left(\tilde{v}_{\cdot j}^{\prime}-\tilde{v}_{\cdot j}\right) \\
            =& \sum_{v \in \mathcal{C}_{react}}(\tilde{x}^*)^{v}\sum_{\{j~:~\tilde{v}_{\cdot j}=v\}}\tilde{u}_j\left(\tilde{v}_{\cdot j}^{\prime}-\tilde{v}_{\cdot j}\right) = \sum_{j=1}^{\tilde{r}}\tilde{u}_j(\tilde{x}^*)^{\tilde{v}_{\cdot j}}\left(\tilde{v}_{\cdot j}^{\prime}-\tilde{v}_{\cdot j}\right)=0
         \end{split}
    \end{equation}
    with the third equality holding due to (\ref{eq_def_LC}). Since $D$ is invertible, we can get $\sum_{j=1}^ru_j\left(x^*\right)^{v_{\cdot j}}\left(v_{\cdot j}^{\prime}-v_{\cdot j}\right)=0$, where $x^*=D^{-1}\tilde{x}^*$. This means that $x^*$ is an equilibrium of $(\mathcal{N},u)$. In addition, by the definition of linear conjugacy, we have $\Phi (x_0,t)=D^{-1}\tilde{\Phi}(Dx_0,t)$. As $\tilde{x}^* \in  (Dx_0+\tilde{\mathscr{S}})\cap \mathbb{R}^n_{>0}$, we have $x^* \in (x_0+\mathscr{S})\cap \mathbb{R}^n_{>0}$.

    \textit{Step 2.} We prove the uniqueness of the equilibrium in $(x_0+\mathscr{S})\cap \mathbb{R}^n_{>0}$. Suppose that $x^{**} \in (x_0+\mathscr{S})\cap \mathbb{R}^n_{>0}$ satisfies $\sum_{j=1}^ru_j\left(x^{**}\right)^{v_{\cdot j}}\left(v_{\cdot j}^{\prime}-v_{\cdot j}\right)=0$. By an argument similar to that in \textit{Step 1}, we obtain $\sum_{j=1}^{\tilde{r}}\tilde{u}_j\left(Dx^{**}\right)^{\tilde{v}_{\cdot j}}\left(\tilde{v}_{\cdot j}^{\prime}-\tilde{v}_{\cdot j} \right)=0$, and $Dx^{**} \in (Dx_0+\tilde{\mathscr{S}})\cap \mathbb{R}^n_{>0}$. Since $(\tilde{\mathcal{N}},\tilde{u})$ is complex balanced, there is a unique equilibrium in each positive stoichiometric compatibility class. Thus we have $Dx^{**}=\tilde{x}^*=Dx^*$, i.e., $x^{**}=x^*$.

    \textit{Step 3.} We establish the smooth dependence property. Since $x^*=D^{-1}\tilde{x}^*$, we can obtain that $x^*$ depends smoothly on $\tilde{x}^*$. According to the \textit{Corollary} 3.13 in \cite{craciun2025structure}, we know that $\tilde{x}^*=\tilde{x}^{*}(\tilde{u})$ depends smoothly on $\tilde{u} \in \mathcal{T}(\tilde{\mathcal{N}})$. From the assumption in \cref{thm_ISS_LC} it follows that $\tilde{u}$ depends smoothly on $u \in \mathcal{U}$. Therefore, we know that $x^*=x^*(u)$ depends smoothly on the parameter values $u \in \mathcal{U}$.
\end{proof}

\begin{proof}[\textbf{Proof of \cref{lem_LCthm_2}}]
    For any compact set $\mathbb{U} \subseteq \mathcal{U}$, denote $V_0(x,u)=\sum_{i=1}^nd_i(x_i \\(\ln x_i-\ln x_i^*-1)+x_i^*),$ where $x^*=x^*(u),~u\in \mathbb{U}$ is defined as in \cref{lem_LCthm_1}.

    According to \cref{lem_LCthm_1}, we know that $x^*=x^*(u)$ depends smoothly on the parameter values $u \in \mathbb{U}$. Therefore, the restriction of $V_0$ to $\mathbb{R}^n_{>0} \times \mathbb{U}$ is continuously differentiable, and the condition (1) in \cref{lem_ISS_Lya} holds.

    For any $u^* \in \mathbb{U}$, it is straightforward to verify that $V_0(x,u^*)\ge 0$, with the equality holding if and only if $x=x^*(u^*)$. Denote
    \begin{equation*}
        \underline{\alpha}_{u^*}(s)=\inf_{|x-x^*|\ge s}V_0(x,u^*),\quad \overline{\alpha}_{u^*}(s)=\sup_{|x-x^*|\le s}V_0(x,u^*),\quad s\ge 0,
    \end{equation*}
    which satisfy $\underline{\alpha}_{u^*}(0)=\overline{\alpha}_{u^*}(0)=0$, and are continuous, positive definite, and strictly increasing (and thus class $\mathcal{K}$ functions). Then it holds 
    $\underline{\alpha}_{u^*}(|x-x^*|) \le V_0(x,u^*) \le \overline{\alpha}_{u^*}(|x-x^*|),~x \in \mathbb{R}^n_{\ge 0}.$ 
    Moreover, since $V_0(x,u^*)\to +\infty$ as $|x-x^*| \to +\infty$, we have $\underline{\alpha}_{u^*}(s)\to +\infty,~\overline{\alpha}_{u^*}(s)\to +\infty$. Therefore, $\underline{\alpha}_{u^*}$ and $\overline{\alpha}_{u^*}$ are class $\mathcal{K}_{\infty}$ functions, and the condition (2) in \cref{lem_ISS_Lya} holds.

    For any $u^* \in \mathbb{U}$, by the proof of \cref{lem_LCthm_1}, there exists a unique equilibrium $x^*=x^*(u^*)$ for the MAS $\dot{x}=f(x;u^*)$ in each positive stoichiometric compatibility class. Moreover, there exists a $\tilde{u}^* \in \mathcal{T}(\tilde{\mathcal{N}})$ such that $x^*=D^{-1}\tilde{x}^*$, where $\tilde{x}^*=\tilde{x}^*(\tilde{u}^*)$ is the complex balanced equilibrium of system $\dot{\tilde{x}}=\tilde{f}(\tilde{x};\tilde{u}^*)=\sum_{j=1}^{\tilde{r}}\tilde{u}_j\tilde{x}^{\tilde{v}_{\cdot j}}(\tilde{v}_{\cdot j}^{\prime}-\tilde{v}_{\cdot j})$. As $\tilde{x}^*$ is the complex balanced equilibrium, let $\tilde{V}_0(\tilde{x},\tilde{u}^*)=\sum_{i=1}^n(\tilde{x}_i(\ln \tilde{x}_i-\ln \tilde{x}_i^*-1)+\tilde{x}_i^*)$, then it follows $\nabla^{\top}_{\tilde{x}}\tilde{V}_0(\tilde{x},\tilde{u}^*)\tilde{f}(\tilde{x};\tilde{u}^*)\le 0$, with the equality holding if and only if $\tilde{x}=\tilde{x}^*$. Since
    \begin{equation*}
        \begin{split}
            &\nabla_{x}^{\top}V_0(x,u^*)=\left(d_1(\ln x_1-\ln x_1^*), ... ,d_n(\ln x_n-\ln x_n^*) \right) \\
            =&(\ln (d_1x_1)-\ln \tilde{x}_1^*,..., \ln (d_nx_n)-\ln \tilde{x}_n^*)~D =\nabla_{\tilde{x}}^{\top}\tilde{V}_0(Dx,\tilde{u}^*)~D
        \end{split}
    \end{equation*}
    and 
    \begin{equation*}
        \begin{split}
            f(x;u^*)&=\sum_{j=1}^ru_j^*x^{v_{\cdot j}}\left(v_{\cdot j}^{\prime}-v_{\cdot j}\right)=\sum_{v \in \mathcal{C}_{react}}x^v \sum_{\{j|v_{\cdot j}=v\}}u_{j}^*\left(v_{\cdot j}^{\prime}-v_{\cdot j}\right) \\
            & =\sum_{v \in \mathcal{C}_{react}}x^v D^{-1}\sum_{\{j|\tilde{v}_{\cdot j}=v\}}\left(\tilde{u}_j^*\prod_{i=1}^nd_i^{v_i}\right)\left(\tilde{v}_{\cdot j}^{\prime}-\tilde{v}_{\cdot j}\right) \\
            &= D^{-1}\sum_{v \in \mathcal{C}_{react}}(Dx)^v \sum_{\{j|\tilde{v}_{\cdot j}=v\}}\tilde{u}_j^*\left(\tilde{v}_{\cdot j}^{\prime}-\tilde{v}_{\cdot j}\right) \\
            & = D^{-1}\sum_{j=1}^{\tilde{r}}\tilde{u}_j^*(Dx)^{\tilde{v}_{\cdot j}}\left(\tilde{v}_{\cdot j}^{\prime}-\tilde{v}_{\cdot j}\right)= D^{-1}\tilde{f}(Dx;\tilde{u}^*)
        \end{split}
    \end{equation*}
    we thus have $\nabla_{x}^{\top}V_0(x,u^*)f(x;u^*)=\nabla_{\tilde{x}}^{\top}\tilde{V}_0(Dx,\tilde{u}^*)\tilde{f}(Dx;\tilde{u}^*)\le 0$, with the equality holding if and only if $Dx=\tilde{x}^*$, i.e., $x=D^{-1}\tilde{x}^*=x^*$, and the condition (3) in \cref{lem_ISS_Lya} holds.
\end{proof}



\bibliographystyle{siamplain}
\bibliography{references}

\clearpage

\thispagestyle{plain}

\null
\begin{center}
{\bfseries\MakeUppercase{Numerical Simulations}\par}
\end{center}
\par
\vskip .11in

To observe the ISS behavior and the parallel computation result more intuitively, we make some numerical simulations for the examples presented in this paper.

\cref{FigCB} shows the numerical simulation results for the system (\ref{eq_ex_CB}) with two kinds of different inputs starting from the same initial point $x_0=(3.7,1.4,0.2)^{\top}$. As can be seen, for the first kind of input (converging to $u^*$), the state will converge to $x^*$, while for the second kind of input (oscillating along $u^*$), the long-term behavior of the state is controlled by a function of $|u-u^*|$.

\begin{figure}[ht]
    \centering
    \includegraphics[width=0.99\textwidth]{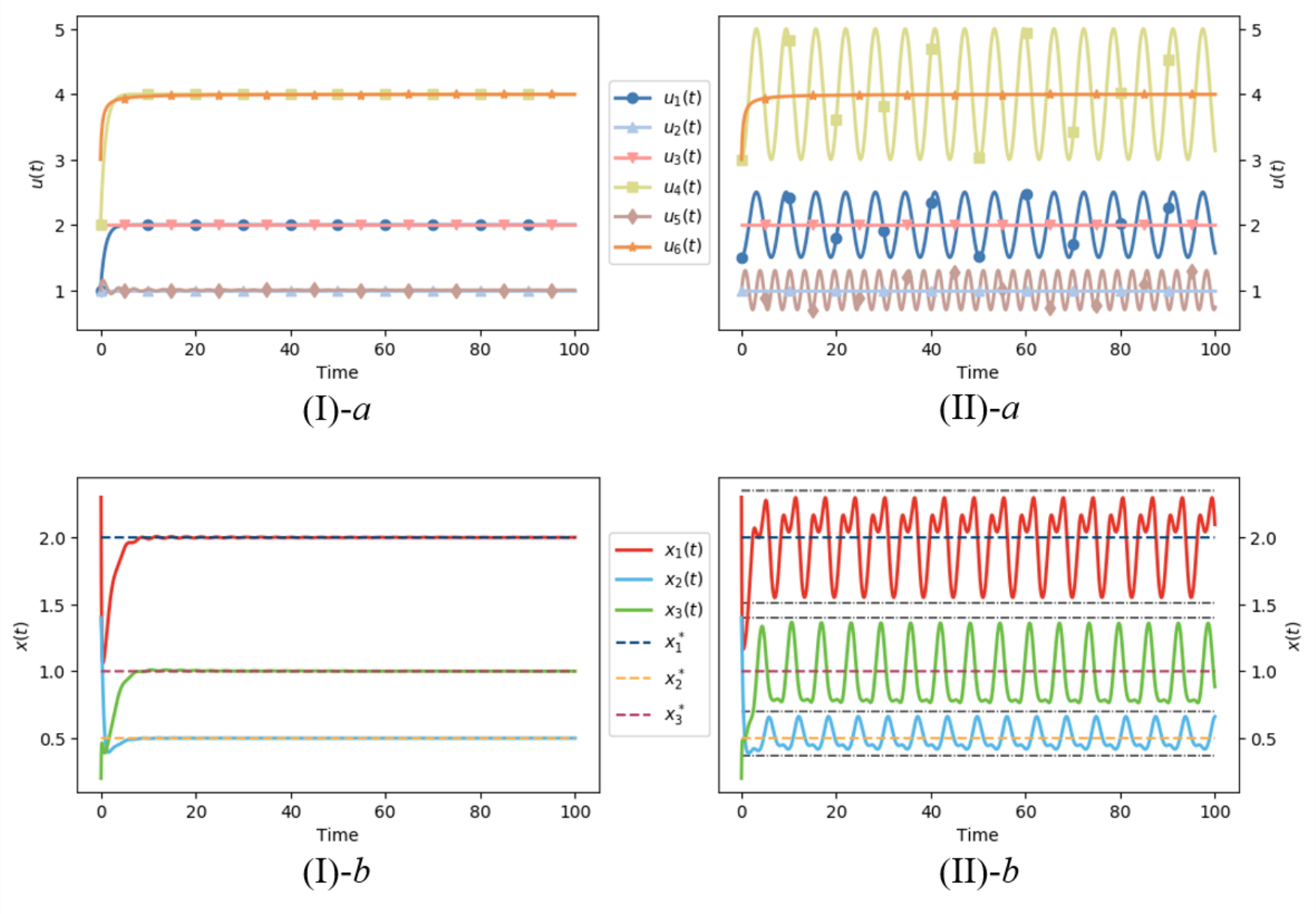}
    \caption{ISS exhibition of (\ref{eq_ex_CB}) with different inputs: (I) $u_1(t)=2-e^{-t}, u_2(t)=1,u_3(t)=2,u_4(t)=4-2e^{-t},u_5(t)=1+\frac{\sin (2t)}{2+7t},u_6(t)=4-\frac{1}{1+3t}$, which converge to $u^*$; (II) $u_1(t)=2-0.5\cos t, u_2(t)=1,u_3(t)=2, u_4(t)=4-\cos t, u_5(t)=1+0.3\sin (2t), u_6(t)=4-\frac{1}{1+3t}$, which oscillate along $u^*$.}
    \label{FigCB}
\end{figure}

\cref{FigBio} presents the numerical simulation results for the system (\ref{eq_supp_ex}) with two types of different inputs starting from the same initial point $x_0=(1.9,1.3,0.2)^{\top}$. Also, for the first kind of input (converging to $u^*$), the state will converge to $x^*$, while for the second kind of input (oscillating along $u^*$), the long-term behavior of the state is controlled by a function of $|u-u^*|$. 

\begin{figure}[ht]
    \centering
    \includegraphics[width=0.99\textwidth]{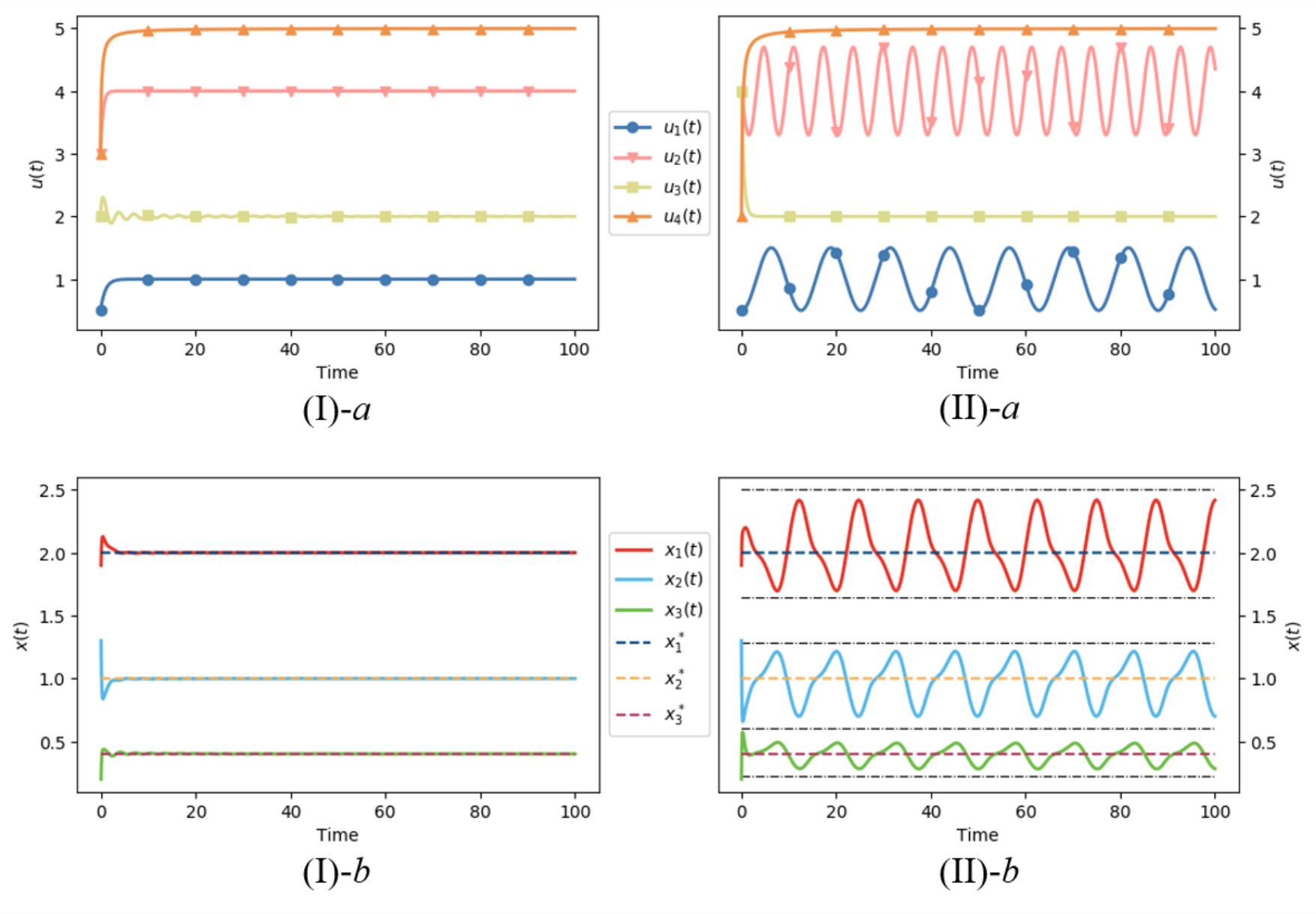}
    \caption{ISS exhibition of (\ref{eq_supp_ex}) with different inputs: (I) $u_1=1-0.5e^{-t}, u_2=4-e^{-2t}, u_3=2+\frac{2\sin(2t)}{2+7t}, u_4=5-\frac{2}{1+5t}$, which converge to $u^*$; (II) $u_1=1-0.5\cos(0.5t), u_2=4-0.7\sin (t), u_3=2+2e^{-2t}, u_4=5-\frac{3}{1+5t}$, which oscillate along $u^*$.}
    \label{FigBio}
\end{figure}

\cref{FigCP} gives the numerical simulation results for the system (\ref{eq_coupled_system}). It can be observed that both $x_1$ and $x_2$ converge to their corresponding concentrations, verifying the effectiveness of parallel molecular computation.

\begin{figure}[ht]
    \centering
    \includegraphics[width=0.99\textwidth]{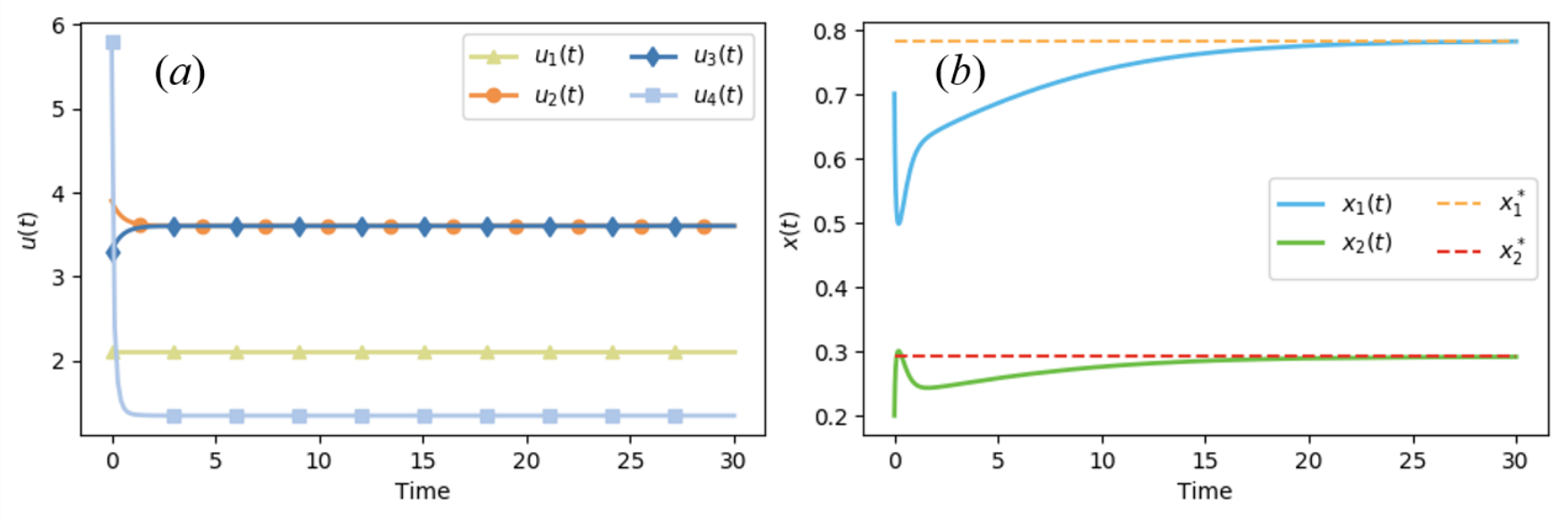}
    \caption{Simulation on the parallel computation result of (\ref{coulped_app}): (a) the evolution of $u(t)$ with $u(0)=(2.1,3.9,3.3,5.8)^{\top}$; (b) the evolution of $x(t)$ with $x(0)=(0.7,0.2)^{\top}$.}
    \label{FigCP}
\end{figure}

\end{document}